\RequirePackage{fix-cm}
\documentclass[11pt, two side, letterpaper, final]{amsart}

\usepackage{euler}
\usepackage{tgpagella}
\usepackage[T1]{fontenc}
\usepackage[utf8]{inputenc}
\usepackage[final]{microtype}

\usepackage{caption}
\usepackage{graphicx}
\usepackage[all,cmtip]{xy}
\usepackage{color}
\usepackage{transparent}

\usepackage{CommonMath}

\usepackage{import}

\usepackage{enumitem}
\usepackage{cleveref} 
\Crefname{hypothesis}{Hypothesis}{Hypotheses}
\Crefname{claim}{Claim}{Claims}
\crefname{claim}{claim}{claims}
\Crefname{maintheorem}{Main Theorem}{main theorem}
\crefname{construction}{construction}{constructions}
\crefname{conjecture}{conjecture}{conjectures}

\newcommand{\comment}[1]{}

\newcommand {\D}{{\bf D}}
\renewcommand{\P}{{\bf P}}
\newcommand {\from}{\colon}
\newcommand{\into}{{\hookrightarrow}}

\newcommand{\sh}[1]{{\mathcal #1}}        
    
\DeclareMathOperator{\GL}{GL}

\DeclareMathOperator{\spec}{Spec}

\DeclareMathOperator{\Sym}{Sym}

\DeclareMathOperator{\Ext}{Ext}

\DeclareMathOperator{\rk}{rk}

\DeclareMathOperator{\Proj}{Proj}

\DeclareMathOperator{\Hom}{Hom}

\DeclareMathOperator{\End}{End}
\DeclareMathOperator{\Sing}{Sing}

\renewcommand{\H}{{\mathcal{H}}}
\newcommand{\cO}{\mathcal{O}}
\newcommand{\cE}{E}
\newcommand{\cN}{N}

\renewcommand{\phi}{\varphi}
\renewcommand{\epsilon}{\varepsilon}

\newcommand{\MC}{MC}
\newcommand{\chain}[1]{\P_{#1}}
\graphicspath{{images/}}

\title{Invariants of a general branched cover of $\P^1$}
\author{Gabriel Bujokas \& Anand Patel}

\address{Department of Mathematics, Harvard University, Cambridge, MA 02138}
\email{gbujokas@math.harvard.edu}

\address{Department of Mathematics, Boston College, Chestnut Hill, MA 02467}
\email{anand.patel@bc.edu}
\date{\today}

\begin{document}
\begin{abstract}
We investigate the resolution of a \emph{general} branched cover $\alpha \from C \to \P^1$ in its relative canonical embedding $C \subset \P E$.  We conjecture that the syzygy bundles appearing in the resolution are balanced for a general cover, provided that the genus is sufficiently large compared to the degree. We prove this for the \emph{Casnati-Ekedahl} bundle, or \emph{bundle of quadrics} $F$---the first bundle appearing in the resolution of the ideal of the relative canonical embedding. Furthermore, we prove the conjecture for all syzygy bundles in the resolution when the genus satisfies $g =  1 \mod d$.
\end{abstract}
\maketitle


\section{Introduction}
\label{sec:introduction}
Every degree $d$ branched cover $\alpha \from C \to \P^1$ from a genus $g$ curve $C$ has a natural  \emph{relative canonical embedding} $$\iota \from C \hookrightarrow \P E$$ into a projective bundle $\P E$ over $\P^{1}$, where $E$ is the classical \emph{Tschirnhausen bundle} -  the dual of the sheaf of traceless functions on $C$, viewed as an $\cO_{\P^1}$-module.  Geometrically, the scroll $\P E$ is swept out by the spans $$\langle \alpha^{-1}(t) \rangle \subset \P^{g-1}$$ for all $t \in \P^{1}$.  The embedding $\iota$ is such that the projection $\pi \from \P E \to \P^1$ induces the map $\alpha$ on $C$, as illustrated in \cref{figure:cover-in-canonical-scroll}.

\begin{figure}[t]
\begin{center}
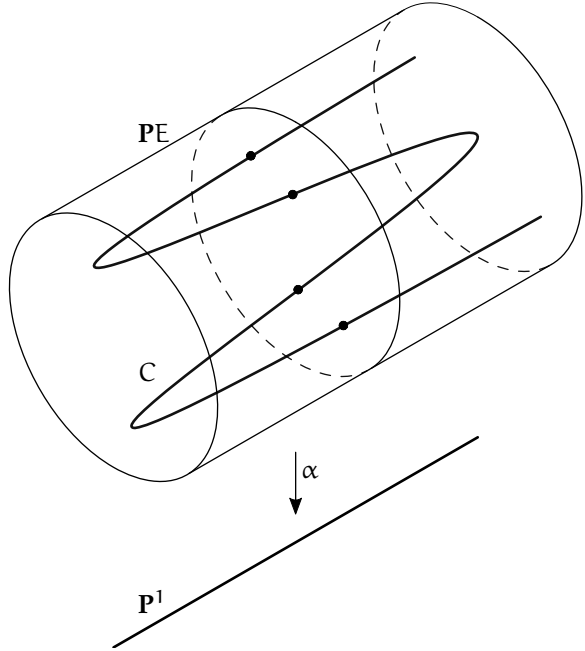
\caption{A cover sitting in its relative canonical scroll.}
\label{figure:cover-in-canonical-scroll}
\end{center}
\end{figure}

The ideal sheaf $\sh I_{C}$ of $C \subset \P E$ has a well known relative minimal resolution, due to a theorem of Schreyer \cite{schreyer_syzygies_1986}, generalized later by Casnati and Ekedahl \cite{casnati_covers_1996}. This resolution of $C \subset \P E$ involves  \emph{syzygy bundles}  $N_{i}$, $i=1, ... , d-2$, which are vector bundles on the target $\P^1$. The splitting types of $N_{i}$ are then fundamental algebro-geometric invariants naturally attached to any branched cover $\alpha \from C \to \P^1$.

Since the set of all branched covers of fixed degree and genus forms the irreducible \emph{Hurwitz space} $\H_{d,g}$, it follows that the syzygy bundles $N_{i}$ have a generic isomorphism type for a Zariski open subset of covers. Our general objective is to understand the generic splitting types of the syzygy bundles.

It is reasonable to expect that the bundles $N_{i}$ are \emph{balanced} for a general cover, in the sense that the largest degree of a summand of $N_{i}$ is at most one more than the smallest degree.  This is succinctly written as 
\[h^{1}(\End N_{i}) = 0.\]

Being balanced is certainly an open condition, in the sense that  in a ``typical''  family of vector bundles on $\P^{1}$, one expects an open set of bundles to be balanced.  With this expectation in mind, and with the aid of explicit examples in low genera and degree, we make the following conjecture:
\begin{conjecture}\label{conj:gen_syzygy}
The general cover $[\alpha \from C \to \P^1] \in \H_{d,g}$ has balanced syzygy bundles when $g \gg d$. 
\end{conjecture}

\begin{remark}
Before considering the syzygy bundles $N_{i}$, one might ask about the generic behavior of the Tschirnhausen bundle $E$.  A theorem of Coppens \cite{coppens_existence_1999} and Ballico \cite{ballico_remark_1989} shows that $E$ is balanced for a generic cover.
\end{remark}

\begin{remark}
\autoref{conj:gen_syzygy} bears a striking resemblence to Green's conjecture on syzygies of canonical curves.  It seems difficult to precisely identify the precise relationship between the two conjectures. \cite{Bopp-comm}.
\end{remark}

This paper focuses specifically on the first bundle $N_{1}$, which we refer to simply as $F$.   $F$ is known as the \emph{bundle of quadrics} or the \emph{Casnati-Ekedahl bundle} - for simplicity we will often  call it the $F$-bundle.  The geometric interpretation of $F$ is very simple: $F$ parametrizes the vector space of quadrics containing the $d$ points of $C$ in the $\P^{d-2}$ fibers of $\P E$.

In recent work of Bopp and Hahn \cite{bopp_syzygies_2014}, the precise necessary relationship between $g$ and $d$ is studied within the regime where the Brill-Noether number satisfies $\rho(g,d,1) \geq 0$ and $g > d+1$.

In this $(g,d)$-regime, they conclude that the generic bundle of quadrics $N_{1}$ is balanced if and only if $\rho > 0$ and $(k- \rho - \frac{7}{2})^{2} - 2d + \frac{23}{4} > 0$. 

Furthermore, with the use of their \emph{Macaulay2}-Package \cite{bopp_macaulay2_2015}, their experimentation leads them to make the following refinement of \cref{conj:gen_syzygy}: 
\begin{conjecture}\label{conj:ref_gen_syzygy}(Bopp, Hahn \cite{bopp_syzygies_2014})
\begin{enumerate}
\item If $\rho \leq 0$ then the bundle $N_{1}$ is balanced for a general cover $[\alpha \from C \to \P^{1}] \in \H_{d,g}$.
\item If $\rho \geq 0$, and if we write $N_{i} = \oplus \cO_{\P^{1}}(a^{(i)}_{j})$, then the following bound holds for a general $\alpha \in \H_{d,g}$: 
\[\max_{j,l}|a^{(i)}_{j} - a^{(i)}_{l}| \leq \min\{g-d-1, i+2\}. \]
\end{enumerate}

\end{conjecture}

\begin{remark}
Notice that the $g = d+1$ case of part $(2)$ is an instance of  \cref{prop:dkcurves}.
\end{remark}

  Our main result is the proof of the $N_{1}$ case of \autoref{conj:gen_syzygy}:
\begin{maintheorem}\label{thm:gen_CE}
Let $g \geq (d-3)(d-1)$. Then the bundle of quadrics $F$ of a general cover $\alpha \in  \H_{d,g}$ is balanced.
\end{maintheorem}

The proof proceeds via degeneration.  We begin in the first section by reviewing the Casnati-Ekedahl structure theorem, and then studying the bundle of quadrics for covers of genus $0$ and $1$. We then glue such covers together in a ``chain'' to create an admissible cover $\alpha \from X \to P$ where $P$ is a chain of $\P^1$'s.

Once we have constructed this cover, we must argue that its $F$-bundle, a vector bundle on the chain of rational curves $P$, is balanced.
In order to do this, we first provide a general criterion for a vector bundle $V$ on a chain $P$ to satisfy $h^{1}(\End V) = 0$.
\Cref{sec:vector-bundles-on-rational-chains} is devoted to establishing this criterion.  

We are then left with the task of verifying that these conditions are met for the $F$-bundle of $\alpha \from X \to P$.  This, in turn, leads us to considering a degree $2$ maximal rank problem for ``maximally connected chains'' of rational normal curves in projective space.  \Cref{sec:reduction} is devoted to establishing the connection between our original problem of showing $h^{1}(\End F) = 0$ with the degree $2$ maximal rank problem.  
We settle the aforementioned maximal rank problem in \cref{sec:the_maximal_rank_problem_for_maximally_connected_chains}.

Throughout, we will work over an algebraically closed field $k$ of characteristic zero.

\subsection*{Acknowledgments}
We would like to thank Christian Bopp for interesting conversations and for sharing his joint work with Michael Hahn.  We thank Eric Larson for suggesting \cref{construction:eric}.   We also had several conversations with Anand Deopurkar, Joe Harris and Eric Riedl which were very helpful. We also thank Carolina Yamate for her crucial assistance in producing the pictures for this paper.


\section{Preliminaries}
\label{sec:preliminaries}
We review the basic Casnati-Ekedahl structure theorem of branched covers of algebraic varieties. We then examine the Casnati-Ekedahl resolution of covers $\alpha \from C \to \P^1$ when the genus $g(C)$ is $0$ or $1$. 
\subsection{The Casnati-Ekedahl Resolution}
 Let $X$ and $Y$ be integral schemes and $\alpha \from X \to Y$ a finite flat Gorenstein morphism of degree $d \geq 3$. The map $\alpha$ gives an exact sequence
\begin{equation}\label{structure sheaf sequence}
  0 \to \cO_Y \to \alpha_* \cO_X \to {\cE_\alpha}^\vee \to 0,
\end{equation}
where $\cE = \cE_\alpha$ is a vector bundle of rank $(d-1)$ on $Y$, called the \emph{Tschirnhausen bundle} of $\alpha$. Denote by $\omega_\alpha$ the dualizing sheaf of $\alpha$. Applying $\Hom_Y(-,\cO_Y)$ to \eqref{structure sheaf sequence}, we get
\begin{equation}\label{dual structure sheaf sequence}
  0 \to \cE \to \alpha_* \omega_\alpha \to \cO_Y \to 0.
\end{equation}
The map $\cE \to \alpha_* \omega_\alpha$ induces a map $\alpha^* \cE \to \omega_\alpha$. 

\begin{theorem}\label{thm:CE}
  \cite[Theorem~2.1]{casnati_covers_1996}
  In the above setup, $\alpha^* \cE \to \omega_\alpha$ gives an embedding $\iota \from X \to \P \cE$ with $\alpha  = \pi \circ \iota$, where $\pi \from \P \cE \to Y$ is the projection. Moreover, the subscheme $X \subset \P \cE$ can be described as follows.
\begin{enumerate}[label=(\alph*)]
\item The resolution of $\cO_X$ as an $\cO_{\P \cE}$-module has the form
\begin{equation}\label{eqn:casnati_resolution}
  \begin{split}
    0 \to \pi^* \cN_{d-2} (-d) \to \pi^* \cN_{d-3}(-d+2) \to \pi^*\cN_{d-4}(-d+3) \to \dots \\
    \dots \to \pi^*\cN_2(-3) \to \pi^*\cN_1(-2) \to \cO_{\P \cE} \to \cO_X \to 0,
  \end{split}
\end{equation}
where the $\cN_i$ are vector bundles on $Y$. Restricted to a point $y \in Y$, this sequence is the minimal free resolution of length $d$ zero dimensional scheme $X_y \subset \P \cE_y\cong \P^{d-2}$.
\item \label{thm:CE item rank} The ranks of the $\cN_i$ are given by
  \[ \rk \cN_i = \frac{i(d-2-i)}{d-1} \binom{d}{i+1},\]
\item We have $\cN_{d-2} \cong \det \cE$.
\item Furthermore, the resolution is symmetric, that is, isomorphic to the resolution obtained by applying the functor
$\Hom_{\cO_{\P \cE}}(-, \pi^* \cN_{d-2}(-d))$.
\end{enumerate}
\end{theorem}

\begin{remark}
We will only use the case $Y=\P^1$, which was discovered by  Schreyer in \cite[Corollary 4.4]{schreyer_syzygies_1986}.
\end{remark}

The branch divisor of $\alpha \from X \to Y$ is given by a section of $(\det \cE)^{\otimes 2}$. In particular, if $X$ is a curve of (arithmetic) genus $g$, $\alpha$ has degree $d$, and $Y = \P^1$, then 
\[
 \rk \cE = d-1 \text{ and } \deg \cE = g+d-1.
 \]

We will be concerned with the vector bundle $N_{1}$, which from here onwards we denote by $F$, and refer to as the \emph{bundle of quadrics}.

Notice that, by twisting the Casnati-Ekedahl resolution \eqref{eqn:casnati_resolution} by $\cO_{\P E}(2)$ and then applying $\pi_{*}$, we obtain an exact sequence 
\begin{equation}\label{Fsequence}
0 \to F \to \Sym^{2}E \to \alpha_{*}(\omega_{\alpha}^{\otimes 2}) \to 0.
\end{equation}
From this sequence, it is easy to see that $\rk F = d(d-3)/2$ and $\deg F = (d-3)(g+d-1)$.

\Cref{thm:CE} allows us to associate to every cover $[\alpha \from C \to \P^1] \in \H_{d,g}$ a collection of vector bundles $(E, F, ... , N_{d-2})$ on $\P^1$. \Cref{conj:gen_syzygy} expresses the natural expectation that these vector bundles are all \emph{balanced} for a general cover $\alpha$ (we say a vector bundle $V$ is balanced if $h^{1}(\End V) = 0$).

To fix notation, we let $\zeta$ denote the divisor class associated to the line bundle $\cO_{\P E}(1)$ on a projective bundle $\P E$.  We consider $\P E$ to be the scheme $\Proj (\Sym^{*}E)$, as in \cite{hartshorne_algebraic_1977}.  We let $f$ denote the class of a fiber of the projection $\pi \from \P E \to \P^1$.

\subsection{Evidence} We explain the rationale behind \cref{conj:gen_syzygy}.

 First, when the degree is at most  $5$, it is known that $F$ is generically balanced:  \cite{deopurkar_picard_2015,bopp_syzygies_2014}.   In these cases, $F$ is (up to twist) the only bundle appearing in the Casnati-Ekedahl resolution, due to the symmetry of the resolution. 
 
 When $d=6$ and $g=4$, the $F$-bundle is generically \emph{not} balanced: 
\begin{example}\label{6,4}
Let $\alpha \from C \to \P^1$ be a general degree $6$ cover, with $C$ a genus $4$, 
non-hyperelliptic curve.  Then $E =   \cO(1) \oplus \cO(2)^{\oplus 4}$. Since $F$ is a sub vector bundle of $\Sym^{2}E$, we conclude that the degree of any summand of $F$ cannot exceed $4$.  

On the other hand, since $\deg F = 27$ and $\rk F = 9$,  the bundle $F$ is imbalanced if and only if $\cO(4)$ is a summand of $F$.

We show that $H^0(F(-4))\neq 0$ by considering the sequence \eqref{Fsequence}. We want
to show that the map $H^0(\P^1, \Sym^2E(-4))  \to H^0( \P^1, \alpha_*(\omega_\alpha) (-4))$ is not injective. But we can identify this map with:
\begin{equation}
\label{eq:map}
H^0(\P E, 2\zeta-4f) \to H^0(C, (2\zeta-4f)|_C)
\end{equation}

The linear system $|\zeta-2f|$ on $\P E$ restricts to the full canonical series on $C$.  Furthermore, the sections of $2\zeta- 4f$ are obtained by taking sums of products of sections of $|\zeta - 2f|$.  Since the canonical model of $C$ lies on a unique quadric in $\P^3$, we see that there is a unique element of $|2\zeta- 4f|$ containing $C$.  This means that the map \eqref{eq:map} is not injective, which implies that $F$ contains an $\cO(4)$ summand. Therefore $F$ is not balanced.
\end{example}

The above example suggests that when the genus is small compared to the degree, the $F$-bundle may be imbalanced. This example can be generalized as long as $g$ is somewhat small compared to $d$, and it explains the inequality $g \gg d$ in 
\cref{conj:gen_syzygy}. 

Substantial support for the conjecture is provided by the following proposition: 

\begin{proposition}\label{prop:dkcurves}
\Cref{conj:gen_syzygy} is true whenever $g = 1 \mod d$. In fact, for a general cover $\alpha \from C \to \P^{1}$, the syzygy bundles $N_{i}$ are \emph{perfectly} balanced, i.e. \[N_{i} = \cO_{\P^{1}}(k_{i})^{\oplus r_{i}}\] for some integers $k_{i}, r_{i}$.
\end{proposition}

We'll need the following basic lemma about elliptic normal curves of degree $d$ in $\P^{d-1}$:
\begin{lemma}\label{lem:ellipticnormal}
Let $E \subset \P^{d-1}$ be a degree $d$ smooth elliptic normal curve, let $H$ be a hyperplane in $\P^{d-2}$, and let $Z := H \cap E$.  
\begin{enumerate}[label=(\alph*)]
\item If 
\begin{equation}\label{resolutionE}
0 \to F_{1} \to ... \to F_{k}\to {\sh I}_{E \subset \P^{d-1}} \to 0
\end{equation}
is the minimal free resolution of $E \subset \P^{d-1}$, then 
\[0 \to F_{1}|_{H} \to ... \to F_{k}|_{H} \to  {\sh I}_{Z \subset H} \to 0\] is the minimal free resolution of $Z \subset H$. 
\item The syzygy bundles $F_{i}$ appearing in the minimal free resolution \eqref{resolutionE} are of the form $\cO_{\P^{d-1}}(-m_{i})^{\oplus r_{i}}$ for some integers $m_{i}$ and $r_{i}$, i.e. all summands of $F_{i}$ have the same degree.
\end{enumerate}

\end{lemma}

\begin{proof}
There is an exact sequence
\[
	0 \to \sh{I}_{E \subset \P^{d-1}}(-H) \to \sh{I}_{E \subset \P^{d-1}} \to \sh{I}_{Z \subset H} \to 0.
\]
Take the minimal resolution of $\sh{I}_{E \subset \P^{d-1}}$:
\[
	0 \to  F_1 \to ... \to F_{k} \to \sh{I}_{E \subset \P^{d-1}} \to 0
\]
And form the following diagram
\[
\xymatrix{
0\ar[r] & F_1(-H)\ar[r] \ar[d] & \ldots\ar[r] & F_k(-H)\ar[r]\ar[d] & \sh{I}_{E \subset \P^{d-1}}(-H)\ar[d]\ar[r] & 0 \\
0\ar[r] & F_1\ar[r] \ar[d] & \ldots\ar[r] & F_k\ar[r]\ar[d] & \sh{I}_{E \subset \P^{d-1}}\ar[d] \ar[r] & 0\\
0\ar[r] & F_1|_H\ar[r]  & \ldots\ar[r] & F_k|_H\ar[r] & \sh{I}_{Z \subset H} \ar[r] & 0\\
}
\]
The top two rows are exact, and so is every column. Therefore, the bottom row is exact as well, and hence a resolution
for $\sh{I}_{Z \subset H}$. The bottom resolution is minimal because the degrees of the $F_{k}$ strictly increase with $k$.  

The second part of the lemma is a well-known computation of the minimal free resolution of an elliptic normal curve found, for example, in \cite{eisenbud_commutative_????}.
\end{proof}

\begin{proof}[Proof of \cref{prop:dkcurves}]
We essentially apply \cref{lem:ellipticnormal} relatively over $\P^{1}$.  Let $C \subset E \times \P^1$ be a smooth bi-degree $(d,m)$ curve; this means that $C$ intersects the elliptic curve rulings $d$ times.  These intersections are all in the same linear system $|D|$ of degree $d$ on $E$.   By adjunction, the curve $C$ is seen to have genus $g = d(m-1) + 1$.

We can embed the surface $E \times \P^1$ into $\P^{d-1} \times \P^1$ where the first factor is embedded by the linear system $|D|$, i.e. $E \times \P^1$ is a constant family of degree $d$ elliptic normal curves in $\P^{d-1} \times \P^1$. 

Let $\pi \from \P^{d-1} \times \P^1 \to \P^1$ denote the projection.  Then the ideal sheaf $\sh I_{E \times \P^1}$ has a ``relative'' minimal resolution which is pulled back via the projection to $\P^{d-1}$:
\begin{equation}\label{resolutionEP}
0 \to F_{1} \to ... \to F_{k}\to {\sh I}_{E \times \P^{1} \subset \P^{d-1} \times \P^{1}} \to 0
\end{equation}

Let $h$ denote the pullback of the hyperplane class on $\P^{d-1}$ to the product $\P^{d-1} \times \P^{1}$, and let $f$ denote the class of a fiber of $\pi$. Then the curve $C \subset E \times \P^{1}$ is the intersection of $E \times \P^{1}$ with some relative hyperplane divisor $\Lambda \subset \P^{d-1} \times \P^{1}$ with divisor class $h+mf$.  

The curve $C$, viewed as a subscheme of the $\P^{d-2}$-bundle $\Lambda$, is a branched cover in its relative canonical embedding. This follows because, by \cref{lem:ellipticnormal}, the fibers $C \cap \P^{d-2} \subset \P^{d-2}$ are length $d$ arithmetically Gorenstein subschemes of $\P^{d-2}$, a property which characterizes the relative canonical embedding \cite{casnati_covers_1996}.

We then restrict sequence \eqref{resolutionEP} to the relative hyperplane $\Lambda$, and use \cref{lem:ellipticnormal} in each fiber of $\pi$ to conclude that 
\begin{equation}\label{resolutionELambda}
0 \to F_{1}|_{\Lambda} \to ... \to F_{k}|_{\Lambda}\to {\sh I}_{C  \subset \Lambda} \to 0
\end{equation}
is the Casnati-Ekedahl resolution of $C \subset \Lambda$. If we write \[F_{i}|_{\Lambda} = \pi^{*}N_{i}(-m_{i})\] we see that each bundle $N_{i}$ is in fact \emph{perfectly} balanced, i.e. \[N_{i} = \cO_{\P^{1}}(k_{i})^{\oplus r_{i}}\] for some integer $k_{i}$, where $r_{i}$ is given in \cref{thm:CE item rank} of \Cref{thm:CE}.  

\end{proof}

\begin{remark}  \Cref{prop:dkcurves} and \cref{lem:ellipticnormal} indicate a direct relationship between invariants of branched covers and properties of genus $1$ fibrations.  

Let $f \from S \to \P^{1}$ be a generically smooth family of arithmetic genus $1$ curves, and let $C \subset S$ be a smooth $d$-section of the map $f$, i.e. $C$ is a smooth curve whose intersection number with the fibers of $S$ is $d$. 

Then the pair $(f \from S \to \P^{1}, C)$ gives rise to the natural embedding 
\begin{equation}\label{embedgenusone}
S \hookrightarrow \P(f_{*}\cO_{S}(C))
\end{equation}
over $\P^{1}$. The projective bundle $\P(f_{*}\cO_{S}(C))$  is a $\P^{d-1}$-bundle, and each fiber of $f$ embeds as a possibly singular degree $d$ elliptic normal curve. Furthermore, there exists a relative hyperplane $\Lambda \subset \P(f_{*}\cO_{S}(C))$ whose intersection with $S$ is precisely the curve $C$.  

The main observation, following from \cref{lem:ellipticnormal}, is that the curve $C$, viewed as a branched cover of $\P^{1}$, is embedded via its relative canonical embedding in the $\P^{d-2}$-bundle $\Lambda$. 

As an analogy, the reader should recall the relationship between the syzygies of a canonical curve $C \subset \P^{g-1}$ and a $K3$ surface $S \subset \P^{g}$ containing $C$.  A generic version of Green's conjecture was proved by Voisin in \cite{voisin_greens_2002} using this relationship.

A dimension count says that a general branched cover $\alpha \from C \to \P^{1}$ does not occur as a $d$-section in a genus $1$ fibration.  However, the proof of \cref{prop:dkcurves} shows us that \emph{if the syzygy bundles for $S \subset \P(f_{*}\cO_{S}(C))$ are balanced, then the same follows for $C \subset \Lambda$.} 

This suggests a more attractive approach for proving \cref{conj:gen_syzygy}:  Show that the generic pair $(f \from S \to \P^{1}, C)$ has balanced syzygies in its relative embedding \eqref{embedgenusone}.  This is the subject of future work.
\end{remark}

\subsection{Covers of genus \texorpdfstring{$0$}{zero} and \texorpdfstring{$1$}{one}.} 
We will now gather some relevant facts about genus $0$ and genus $1$ covers of $\P^1$.

\begin{lemma}\label{lowG}
Let $\alpha \colon R \to \P^1$ and $\beta \from X \to \P^1$ be degree $d$ covers where $R$ is a smooth rational curve and $X$ is a smooth elliptic curve. Then:
\begin{enumerate} [label=(\alph*)]
\item \label{item-E-rational} $E_{\alpha} = \sh O(1)^{\oplus d-1}$,
\item \label{item-F-rational}  $F_{\alpha} = \sh O(1)^{\oplus d-3} \oplus \sh O(2)^{\oplus \binom{d-2}{2}}$
\item \label{item-E-elliptic} $E_{\beta} = \sh O(1)^{\oplus d-2} \oplus \sh O(2)$
\item \label{item-F-elliptic} $F_{\beta} =  \sh O(2)^{\oplus \frac{d(d-3)}{2}}$
\end{enumerate}
\end{lemma}

\begin{proof} 
For \cref{item-E-rational}, note that the sequence
\[
0 \to \sh{O}_{\P^1} \to \alpha_*\sh{O}_C \to E_{\alpha}^\vee \to 0
\]
splits, and therefore $H^0(E_{\alpha}^\vee)=0$. This means all summands of $E_{\alpha}$ are positive and add up to $ d-1$.  Therefore, all summands must have degree one.

For \cref{item-F-rational}, note that using the relative canonical factorization $\iota \from R \to \P E$, we may think of $R$ as lying inside $\P E$. The series $|\zeta - f|$ on $\P E$ restricts to the complete series $\sh O_{R}(d-2)$, and the morphism $q \from \P E \to \P^{d-2}$ given by the series $|\zeta - f|$ restricts to the embedding of $R$ into $\P^{d-2}$ as a rational normal curve. 

The rational normal curve $R$ is contained in a $\binom{d-2}{2}$ dimensional space of quadrics.  The divisor class of these quadrics, when pulled back along $q$, is $2\zeta - 2f$.  

Recall the exact sequence on the target $\P^1$: 
\begin{equation}\label{FEexact}
0 \to F_{\alpha} \to \Sym^{2}E_{\alpha} \to \alpha_{*}{\omega^{2}_{\alpha}} \to 0.
\end{equation}
We twist by $\sh O(-2)$ and consider global sections, to get:
\begin{align*}
H^0(\P^1, F_{\alpha}(-2)) = & \ker (H^0(\P^1, \Sym^2E_{\alpha} (-2)) \to 
H^0(\P^1, \alpha_{*}{\omega^{2}_{\alpha}}(-2))) \\
= &\ker (H^0(\P E_\alpha, 2\zeta-2f) \to H^0(R, (2\zeta -2f)|_R)) \\
= & \ker (H^0(\P^{d-2}, \sh{O}_{\P^{d-2}}(2)) \to H^0(R,\sh{O}_{\P^{d-2}}(2)|_R ))  
\end{align*}

The previous paragraph along with projective normality of $R \subset \P^{d-2}$ imply that $h^{0}(F_{\alpha}(-2)) = \binom{d-2}{2}$.  We conclude by noting that \eqref{FEexact} shows that no degree of a summand of $F_{\alpha}$ may exceed $2$, and the degree of $F_{\alpha}$ must be $(d-3)(d-1)$.  This forces its splitting type to be $\sh O(1)^{\oplus d-3} \oplus \sh O(2)^{\oplus \binom{d-2}{2}}$.

For \cref{item-E-elliptic}, we note that all $d-1$ summands of $E_{\beta}$ are positive, and their degrees sum to $d$.  Therefore $E_{\beta}$ must be as indicated.

Finally, for \cref{item-F-elliptic}, see that the pencil of sections of $\beta^{*}\sh O_{\P^{1}}(1)$ on $X$ forms a two dimensional vector subspace of its complete linear system.  The complete series gives an embedding of $X$ into the projective space $\P^{d-1}$, and the map $\beta$ is realized as projection from a general $(d-3)$-dimensional linear space $\Lambda \subset \P^{d-1}$ disjoint from $X$. 

We note that $\P E$, as an abstract scroll, is isomorphic to $Y := Bl_{\Lambda}\P^{d-1}$. The linear system $|\zeta - f|$ provides the map $f \from \P E \to \P^{d-1}$.  Furthermore, the linear system $|2\zeta - 2f|$ parametrizes the quadric hypersurfaces in $\P^{d-1}$.  Therefore, $h^{0}(F_{\beta}(-2))$ is simply the vector space dimension of quadrics in $\P^{d-1}$ containing the elliptic normal curve $X \subset \P^{d-1}$, which is easily calculated to be $d(d-3)/2$. 

Next we show that $h^{0}(F_{\beta}(-3)) = 0$. An element of the linear system $|2\zeta - 3f|$ corresponds to a quadric in $\P^{d-1}$ which splits off a  hyperplane component $\Gamma$ containing $\Lambda$.  No such quadric can contain the elliptic normal curve $X$, by nondegeneracy of the curve.  This means  $h^{0}(F_{\beta}(-3)) = 0$.  

Therefore, the largest degree summand of $F_{\beta}$ is $\cO(2)$. Since the degree of $F_{\beta}$, by sequence \eqref{FEexact}, is $d(d-3)$ we conclude that $F_{\beta}$ must split as $\cO(2)^{\oplus \frac{d(d-3)}{2}}$.
\end{proof}
\begin{remark}
Of course \cref{lowG} part $(d)$ also follows from the much more general \cref{prop:dkcurves}. 
\end{remark}


\section{Vector bundles on rational chains}
\label{sec:vector-bundles-on-rational-chains}
In this section we find necessary and sufficient conditions for determining when a vector bundle on a chain of rational curves is balanced.  We will eventually use these criteria to prove \cref{theorem:F-bundle}.  

Let $\chain{} = P_{1} \cup P_{2} \cup \ldots \cup P_{k}$ be a chain of $k$ rational curves $P_{i}$, and let $V$ be a vector bundle on $\chain{}$ of rank $r$.  We let $V_{i} = V|_{P_{i}}$ be the restriction to the $i$-th component, and we denote by $p_{i}$ the node $P_{i} \cap P_{i+1}$, for $i=1,\ldots,k-1$.  
\begin{definition}
A vector bundle $V$ on a reduced curve $C$ is \emph{balanced} if $$h^{1}(\End V) = 0.$$ 
\end{definition}
Since the first order deformations of a vector bundle $V$ are parametrized by $\Ext^1(V,V) = H^{1}(\End V)$, we see that a balanced vector bundle does not deform.  Our goal is to determine necessary and sufficient criteria for a vector bundle $V$ on the chain $\chain{}$ to be balanced in terms of the ``relative position'' of the vector bundles $V_{i}$. 

Our first easy observation is the following.
\begin{lemma}
\label{lemma-balanced-components}
Let $V$ be a vector bundle on $\chain{}$ of rank $r$. If $V$ is balanced, then every component $V_{i}$ is a balanced vector bundle on $P_{i}$.
\end{lemma} 
\begin{proof}
Let $\nu \from \coprod_{i} P_{i} \to \chain{}$ be the total normalization map. Consider the exact sequence of $\sh{O}_{\chain{}}$-sheaves 
\begin{equation}\label{norm}
0 \to \End V \to \nu_{*}\nu^{*}(\End V) \to \oplus_{i} \End (k(p_{i})^{r}) \to 0.
\end{equation}
 The last part of the long exact sequence of cohomology looks like \[\ldots \to H^{1}(\chain{}, \End V) \to H^{1}(\chain{}, \nu_{*}\nu^{*}\End V)  \to 0.\]  Since $\nu$ is finite, we can identify the vector space $H^{1}(\chain{}, \nu_{*}\nu^{*}\End V)$ with $H^{1}(\coprod P_{i}, \nu^{*}\End V) = \oplus_{i}H^{1}(P_{i}, \End V_{i})$.  Therefore, if $H^{1}(\chain{}, \End V) = 0$, then $\oplus_{i}H^{1}(P_{i}, \End V_{i}) = 0$, which proves the lemma.
\end{proof}

We now assume that each component $V_{i}$ is balanced.  Consider the long exact sequence of cohomology groups associated to \eqref{norm}: 
\begin{equation}\label{les}
\ldots \to \oplus_{i}H^{0}(P_{i}, \End V_{i}) \to \oplus_{i} \End (k(p_{i})^{r}) \to H^{1}(P, \End V) \to 0
\end{equation}
This sequence shows that $h^{1}(P, \End V) = 0$ if and only if the map \[D \from \oplus_{i}H^{0}(P_{i}, \End V_{i}) \to \oplus_{i} \End (k(p_{i})^{r})\] is surjective.  The map $D$ is described as follows.  Let $\infty_i \in P_{i}$, $0_{i+1} \in P_{i+1}$ denote the two pre images of the node $p_{i}$ in the normalization $\coprod P_{i}$. For any $(M_{1}, \ldots ,M_{k}) \in \oplus_{i}H^{0}(P_{i}, \End V_{i})$,
\[
D(M_{1}, \ldots ,M_{k}) = (M_{1}|_{\infty_{1}} - M_{2}|_{0_{2}}, \ldots , M_{k-1}|_{\infty_{k-1}} - M_{k}|_{0_{k}}).
\]
In other words, since the structure sheaves $k(\infty_i)$ and $k(0_{i+1})$ are naturally identified with $k(p_{i})$ via the normalization map $\nu$, we may compare the restriction of any section $M_{i} \in H^{0}(P_{i}, \End V_{i})$ at $\infty_i$ with the restriction of any section $M_{i+1} \in H^{0}(P_{i+1}, \End V_{i+1})$ at $0_{i+1}$.  For this reason, we call $D$ the \emph{difference} map.  (We are using the letter $M$ to emphasize that we are thinking of these sections as $r \times r$ matrices.)

Recall that the first order deformation space of a vector bundle $V$ is naturally identified with the vector space $H^{1}(\End V)$.  Consider the connecting homomorphism  \[\oplus_{i} \End (k(p_{i})^{r}) \to H^{1}(P, \End V).\]  There is a very simple way of interpreting this map in terms of the deformations of $V$. A vector bundle $V$ on $\chain{}$ is equivalent to the data of the $k$ vector bundles $(V_{1},\ldots, V_{k})$ and gluing maps $(g_{1}, \ldots ,g_{k-1})$, where $g_{i} \from V_{i}|_{\infty_i} \to V_{i+1}|_{0_{i+1}}$ are invertible linear maps.  We can deform this data by simply varying the gluing maps, viewed as invertible $r \times r$ matrices.  The tangent space to $\GL_{r}(k)$ is the space $M_{r \times r}(k)$ of $r \times r$ matrices, and the resulting $k-1$-tuple of of elements in $M_{r \times r}$ correspond to components in the direct sum $\oplus_{i} \End (k(p_{i})^{r})$.  Therefore, by considering different gluing data, we think of the condition of being balanced as a condition about the ``relative position'' of the $V_{i}$ being general.  We will make this more precise in the next subsection.

\subsection{Directrices and Filtrations} A balanced bundle on $\P^1$ has a canonically defined \emph{directrix} subbundle. 
\begin{definition}
Let $V = \sh O(m)^{\oplus a} \oplus \sh O(m+1)^{\oplus b}$ be a balanced vector bundle on $\P^1$.  Then the \emph{directrix} subbundle $W$ is the summand $\sh O(m+1)^{\oplus b}$. 
\end{definition}

Now let us mark two points $0$ and $\infty$ on $\P^1$, and let $[s : t]$ be homogeneous coordinates on $\P^1$ with $0 = [0 :1]$ and $\infty = [1 :0 ]$. Suppose $V$ is a balanced vector bundle on $\P^1$ and $W \subset V$ its directrix subbundle. Of course, there is no canonical identification of $V|_{0}$ with $V|_{\infty}$, since the vector bundle is not trivial. 
However, since $W = \sh O(m+1)^{\oplus b}$, the subbundle $\P W \subset \P V$ \emph{is} trivial.
Therefore, the subspace $W|_{0} \subset V|_{0}$ naturally corresponds to the subspace $W|_{\infty} \subset V|_{\infty}$, \emph{as a subspace}.  Similarly, any linear subspace $L|_{0} \subset W|_{0}$ corresponds to a unique subspace $L_{\infty} \subset W_{\infty}$.

By the same token, since $V/W = \sh O(m)^{\oplus a}$, we see that any subspace $N|_{0} \subset V/W|_{0}$ can be naturally identified with a subspace $N|_{\infty} \subset V/W|_{\infty}$.  To summarize we have established the following.
\begin{lemma}\label{transport}
Let $W \subset V$ be as above.  Then subspaces of $W|_{0}$ are in natural one to one correspondence with subspaces of $W|_{\infty}$.  Furthermore, intermediate subspaces $W|_{0} \subset N_{0} \subset V|_{0}$ are in natural one to one correspondence with intermediate subspaces $W|_{\infty} \subset N_{\infty} \subset V|_{\infty}.$
\end{lemma}

\begin{definition}
Let $F^{\bullet}$ be an increasing filtration of a vector space $A$, and let $B \subset A$ be a subspace.  We say that $F^{\bullet}$ \emph{contains} $B$ if $B = F^{i}$ for some $i$. 
\end{definition}

\begin{corollary}\label{flag_transport}
The filtrations $F^{\bullet}_{0}$ of $V|_{0}$ containing $W|_{0}$ are in natural one to one correspondence with the filtrations $F^{\bullet}_{\infty}$ of $V|_{\infty}$ containing $W|_{\infty}$.
\end{corollary}

We let \[r \from \big\{\text{filt. $F^{\bullet}_{0}$ of $V|_{0}$ containing $W|_{0}$ }\big\} \to \big\{\text{filt. $F^{\bullet}_{\infty}$ of $V|_{\infty}$ containing $W|_{\infty}$}\big\}\] denote the natural correspondence, and let $l$ denote its inverse. We call $r(F^{\bullet}_{0})$ the \emph{right transport} of $F^{\bullet}_{0}$, and $l(F^{\bullet}_{\infty})$ the \emph{left transport} of $F^{\bullet}_{\infty}$

For a balanced vector bundle $V$ with directrix $W \subset V$, we define \[G^{\bullet}(V):= \{0 \subset W \subset V\}\] to be the \emph{directrix flag} of $V$.  Furthermore, we let $G^{\bullet}(V|_{0})$ and $G^{\bullet}(V|_{\infty})$ denote the respective flags in $V|_{0}$ and $V|_{\infty}$.

\begin{definition}\label{refine}
Let $F^{\bullet} = \{0 = F^{0} \subset F^{1} \subset \ldots \subset F^{N} = A\}$ and $G^{\bullet} =\{0  \subset B \subset A\}$ be two filtrations of a $k$-vector space $A$. Then the \emph{modification of $F^{\bullet}$ by $G^{\bullet}$}, denoted $F \vee G$, is the flag with elements $(F \vee G)^{i} = F^{i} \cap B$ for $i \leq N$, and $(F \vee G)^{j} = B+F^{j-N}$ for $j \geq N$.  
\end{definition}

\subsubsection{The two natural filtrations at a node.}
\label{subsub-the-two-natural-filtrations}
Now let $\chain{} := P_1 \cup P_2 \cup \ldots \cup P_k$ be a chain of $\P^1$'s.  Recall that we let $0_{i}$ and $\infty_{i}$ denote marked points $0$ and $\infty$ on the respective $P_{i}$, and $p_{j}$ be the node joining $P_{j}$ with $P_{j+1}$, so that $p_{j} = \infty_{j} = 0_{j+1}$.  

Let $V$ be a rank $r$ vector bundle on $\chain{}$.  We continue to assume that each $V_{i}$ is balanced, and let $W_{i} \subset V_{i}$ be the $i$-th directrix subbundle.  

At each node $p_{j}$ there are two natural filtrations of $V|_{p_{j}}$, called the \emph{left filtration} $L^{\bullet}_{j}$ and the \emph{right filtration} $R^{\bullet}_{j}$, which we describe inductively.  Let $L^\bullet_1 =G^\bullet(V_1|_0)$,
and inductively define  $L^{\bullet}_{j}$ as the filtration given by 
\[
L^{\bullet}_{j} := r\left(L^\bullet_{j-1} \wedge G^\bullet(V_j|_0) \right) \text{.}
\] 
In words, we sequentially right transport and modify the directrix flags of the vector bundles $V_{i}$, beginning at $V_{1}$ and ending at $V_{j}$.  The filtration $R^{\bullet}_{j}$ is defined similarly, except we begin at $V_{k}$ and then sequentially left transport and modify the directrix subbundles until we reach $V_{j+1}$: 
\[
R^\bullet_j=l\left(R^\bullet_{j+1}\wedge G^\bullet(V_{j+1}|_\infty) \right)
\]

The essential point is that the collection of filtrations $L^{\bullet}_{j}$ and $R^{\bullet}_{j}$ completely captures the ``relative position'' of the system of directrices $W_{i} \subset V_{i}$.  

\begin{definition}
We say that a vector bundle $V$ on $\chain{} = P_{1} \cup \ldots \cup P_{k}$ as above \emph{has transverse directrices} if the filtrations $L^{\bullet}_{j}$ and $R^{\bullet}_{j}$ meet transversely for each node $p_{j}$. 
\end{definition}

\begin{figure}
\begin{center}
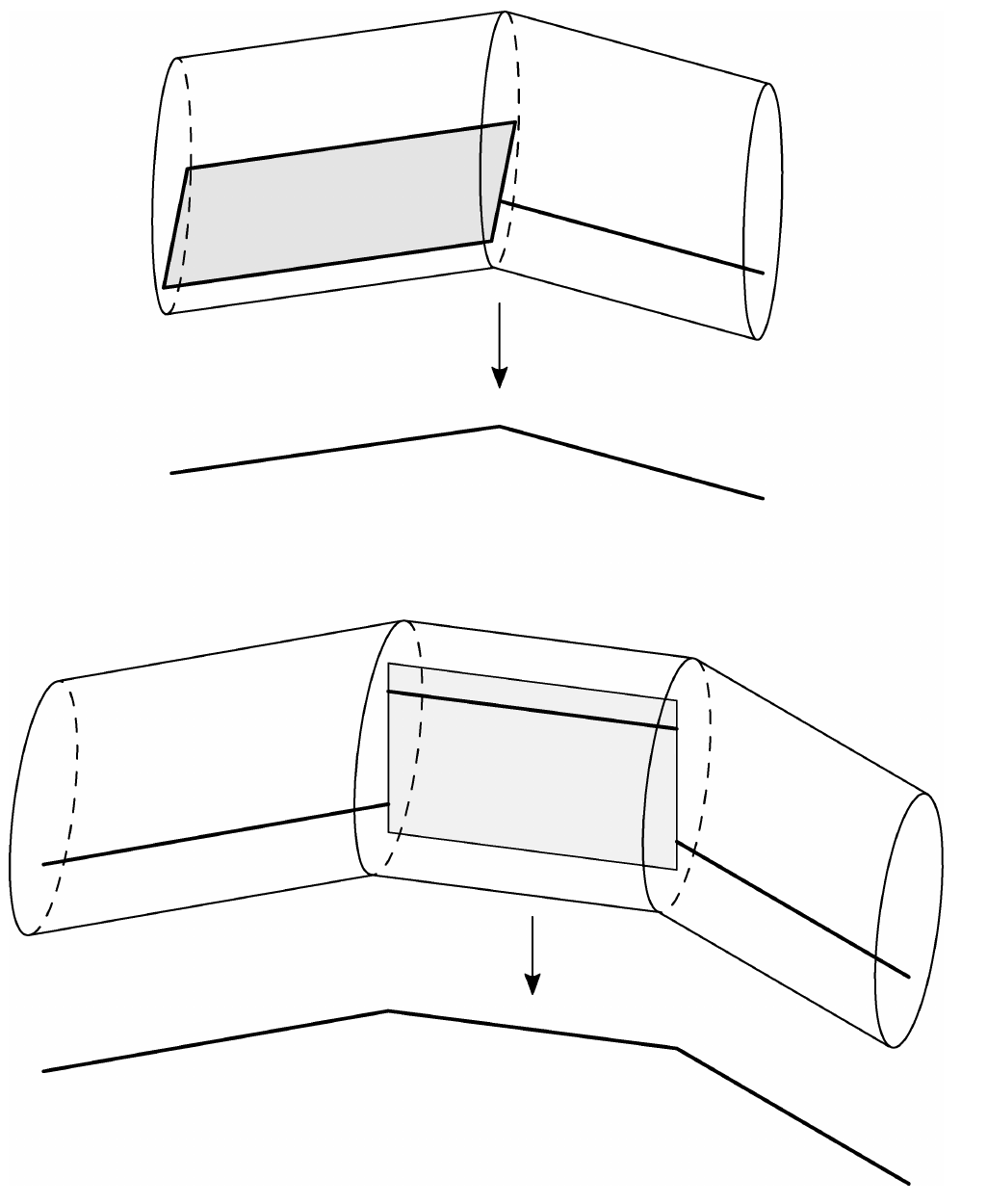
\caption[Some bundles whose directrices fail to meet transversely.]{Some bundles whose directrices fail to meet transversely.

In the first diagram, the codimension 2 directrix $\P W_2$ of $\P V_2$ meets the codimension 1 directrix $\P W_1$ of $\P V_1$
over the node $P_1 \cap P_2$, even though $\P V$ is a $\P^2$-bundle.   

In the second diagram, although adjacent directrices meet transversely, if we transport the flag above node
 $P_1 \cap P_2$ to node $P_2 \cap P_3$ via the minimal subbundle $\P T$, we get non-transverse intersection with
 $\P W_3$. 
}
\label{figure:transverse-directrices}
\end{center}
\end{figure}

See \cref{figure:transverse-directrices} for examples of bundles whose directrices fail to meet transversely. We can now state the theorem which characterizes balanced vector bundles on chains of $\P^{1}$'s.

\begin{theorem}\label{criteria}
Let $V$ be a vector bundle on $\chain{} = P_{1} \cup  \ldots \cup P_{k}$.  Then  $h^{1}(\End V) = 0$ if and only if 
\begin{enumerate}
\item \label{criteria:cond1} Each $V_i$ is balanced, for $i = 1, \ldots, k$, and
\item \label{criteria:cond2} $V$ has transverse directrices. 
\end{enumerate}
\end{theorem}

The remainder of this section will be spent proving \cref{criteria}.

\Cref{lemma-balanced-components} already states that if $V$ is balanced, 
then each $V_i$ is balanced as well. The following proposition shows that $V$ also has to have 
transverse directrices.

\begin{proposition}\label{deform}
Let $V$ be a vector bundle on $\chain{}$ such that $V_{i}$ are balanced for all $i$. Furthermore, suppose there is a node $p_{j}$  where $L^{\bullet}_{j}$ and $R^{\bullet}_{j}$ fail to intersect transversely.  Then $V$ is not balanced, i.e. $V$ admits a non trivial first order deformation.
\end{proposition}

\begin{proof}[Proof of \cref{deform}]
We can produce non-trivial deformations of $V$ by modifying the gluing over $p_j$, in a way that the filtrations
$L^{\bullet}_{j}$ and $R^{\bullet}_{j}$ become transverse. While it is clear that there will be non-trivial deformations over a DVR, we need to check that the restriction to $\D=\spec k[\epsilon]/(\epsilon^2)$ is still
non-trivial. 

We get our hands on the first order deformations from 
the boundary map $$\delta \from \oplus_{i}  H^{0}(\End k(p_{i})^{ r}) \to H^1(\End V).$$
Given $M = (M_{1},  \ldots , M_{k-1}) \in  \oplus _{i} H^{0}(\End k(p_{i})^{ r})$, we construct the corresponding first order deformation $\delta(M)$ of $V$ as follows. Take
the (trivially deformed) bundle $V_i \times \D$ on each component $P_i \times \D$, and glue these vector bundles using the map
\[
	f_i+\epsilon M_i \from \left(V_i|_{\infty} \times \D \right) \to \left(V_{i+1}|_0 \times \D \right) 
\]
where $f_i\from V_i|_\infty \to V_{i+1}|_0$ are the original gluing maps for the vector bundle $V$ on $P$.

We claim that a generic choice of $M$ makes the deformation obtained by the gluings above nontrivial. As a matter of fact, we can even take 
$M_i=0$ for $i\neq j$, as long as $M_j$ \emph{separates} the flags $L^{\bullet}_{j}$ and $R^{\bullet}_{j}$ in
the following sense.

\begin{definition}
\label{definition-separation}
Let $V$ be an $r$-dimensional $k$-vector space, and let $V[\epsilon]$ be the module $V \otimes_{k} k[\epsilon]/(\epsilon^2)$.  Suppose $A$ and $B$ are subspaces of $V$ which fail to meet properly. Let $M$ be an endomorphism of $V$ and consider the invertible module map \[f_{M} := Id + \epsilon M \from V[\epsilon] \to V[\epsilon].\] 

We say that $M$ \emph{separates} $A$ and  $B$ if the submodule $A[\epsilon] \cap f_{M}(B[\epsilon])$ is not flat as a $k[\epsilon]/(\epsilon^2)$-module.
\end{definition}

Note that if $M_j$ separates  the flags $L^{\bullet}_{j}$ and $R^{\bullet}_{j}$, then the deformation is not trivial.
We just have to show that a generic choice of $M_j$ does separate these flags. The following lemma assures that this is the case.

\begin{lemma}
\label{lemma-separate}
Let $V$, $A$, and $B$ be as in \cref{definition-separation}. Then there exists $M \in \End V$ which separates $A$ and $B$.  Furthermore, the general choice of $M$ separates $A$ and $B$.
\end{lemma}

\begin{proof}[Proof of \cref{lemma-separate}]
Note that flat and free are the same notion for the local Artinian ring $k[\epsilon]/(\epsilon^{2})$.  Furthermore, a $k[\epsilon]/(\epsilon^{2})$-module $N$ is flat if and only if the multiplication map 
$\times \epsilon \from \left(N/\epsilon N \right) \to N$ is injective. 

We now choose a basis $\{v_{1}, \ldots , v_{r}\}$ for $V$ such that $A = \langle v_{1}, \ldots ,v_{a}\rangle$ and $B = \langle v_{m+1}, \ldots , v_{m+b}\rangle$, and assume that $m+b < r$ and $m+1 \leq a$, so that $A$ and $B$ do not meet properly.  Now consider the endomorphism $M \in \End V$ which is $0$ on all $v_{i}$ except sends $v_{m+1}$ to $v_{r}$.   Then $f_{M}(B[\epsilon]) = \langle v_{m+1} + \epsilon v_{r}, v_{m+2}, \ldots ,v_{m+b}\rangle \subset V[\epsilon]$.  The element $\epsilon v_{m+1}$ is annihilated by $\epsilon$, but is non-zero in the quotient module $A[\epsilon] \cap f_{M}(B[\epsilon])/\epsilon(A[\epsilon] \cap f_{M}(B[\epsilon]))$. Hence, $A[\epsilon] \cap f_{M}(B[\epsilon])$ is not flat.
\end{proof}
\Cref{lemma-separate} now implies \cref{deform}.
\end{proof}

\Cref{deform} tells us conditions (\ref{criteria:cond1}) and (\ref{criteria:cond2}) are necessary in \cref{criteria}. Now we must prove sufficiency.
Specifically, we must show that if $V$ has transverse directrices, then the difference map
\[
D\from \oplus_{i}H^{0}(P_{i}, \End V_{i}) \to \oplus_{i} \End (k(p_{i})^{r})
\]
is surjective. For $M \in \oplus_{i}H^{0}(P_{i}, \End V_{i})$, denote by $D_i(M)$ the $i$-th component of $D(M) \in \oplus_{i} \End (k(p_{i})^{r})$.
It is enough to show that for each $j$, we can choose $M$ such that $D_i(M)=0$ for all 
$i \neq j$, while  $D_j(M) \in \End (k(p_j)^r)$ is arbitrary.

We will first choose $M_i$ for $i< j$ such that $D_i(M)=0$ for all $i<j$, and see what constraints we have on 
$M_{j-1}|_{\infty}$. We will then do the same from the other side, starting at $M_{k-1}$ and going down to $M_j$, and
then conclude that the difference $D_j(M)=M_{j-1}|_\infty- M_j|_0$ can be made to be arbitrary.

To express the constraints on $M_{j-1}|_{\infty}$, let us introduce the following notation:
\begin{definition}
We say that an endomorphism $M \in \End(V)$ of a vector space $V$ \emph{respects the flag} 
$0 \subset F_1 \subset F_2 \subset \ldots \subset F_r \subset V$ if $M$ preserves the $F_i$ as subspaces of $V$.
\end{definition}

We now can state the following.
\begin{lemma}
\label{lemma-leftside}
We can choose $M_i \in H^{0}(P_{i}, \End V_{i})$ for $i< j$ such that $D_i(M)=0$ for $i<j$,
and $M_{j-1}|_{\infty}\in \End (k(p_j)^r)$ is an arbitrary endomorphism that preserves the left flag 
$L^\bullet_j$.
\end{lemma}
Of course the analogous result is true if we started from the right: we can choose $M_i \in H^{0}(P_{i}, \End V_{i})$ for $i \geq j$ such that $D_i(M)=0$ for $i>j$ and $M_j|_0 \in  H^{0}(P_{i}, \End V_{i})$ is an arbitrary
endomorphism preserving the \emph{right} flag $R^\bullet_j$. Now we only have to assure that we can arrange the difference  $M_{j-1}|_{\infty}-M_j|_0$ to be arbitrary. The following lemma does that. 

\begin{lemma}
\label{lemma-transverse-flags}
If $F^\bullet$ and $G^\bullet$ are transverse flags of a vector space $V$, then any endomorphism of $V$ can be written as a difference of an endomorphism respecting the flag $F^\bullet$, and one respecting $G^\bullet$. 
\end{lemma}
\begin{proof}[Proof of \cref{lemma-transverse-flags}]
Without loss of generality, we may assume that $F^\bullet$ and $G^\bullet$ are complete flags. Indeed, if they are not, just choose generic finer complete flags containing them, and the problem only becomes more restrictive.

Now pick a basis $v_1,\ldots, v_n$ of $V$ such that $\langle v_i \rangle = F^i \cap G^{n-i-1}$. In this basis, 
endomorphisms preserving $F^\bullet$ are upper triangular matrices, and the ones respecting $G^\bullet$ are the lower
triangular ones. And any matrix can be written as a difference of two such matrices.
\end{proof}

To complete the proof of \cref{criteria}, we only need to prove \cref{lemma-leftside}.

\begin{proof}[Proof of \cref{lemma-leftside}]
We proceed by induction on the length of the chain $\chain{}$. If there is only one component, the claim states that for
any endomorphism $\overline{M}:V_1|_\infty \to V_1|_\infty$ which sends the restriction of the directrix $W|_{\infty}$  to itself,
we can find $M \in H^0(\End(V_1))$ such that $M|_\infty = \overline{M}$.

Fix a splitting $V_1 = \sh{O}(m)^a \oplus \sh{O}(m+1)^b$. Then $\End(V_1)= \sh{O}^{a^2} \oplus \sh{O}(1)^{ab} \oplus
\sh{O}(-1)^{ab} \oplus \sh{O}^{b^2}$, and we can realize this splitting more naturally by realizing an element of $\End V_1$ as an $(a+b) \times (a+b)$ matrix block matrix. For
example, for $a=2, b=3$, we have
\[
\End V_1 =
\left[
\begin{array}{cc|ccc}
\sh{O} & \sh{O} & \sh{O}(1) & \sh{O}(1) & \sh{O}(1) \\
\sh{O} & \sh{O} & \sh{O}(1) & \sh{O}(1) & \sh{O}(1) \\ \hline
\sh{O}(-1) & \sh{O}(-1) & \sh{O} & \sh{O} & \sh{O}  \\
\sh{O}(-1) & \sh{O}(-1) & \sh{O} & \sh{O} & \sh{O}  \\
\sh{O}(-1) & \sh{O}(-1) & \sh{O} & \sh{O} & \sh{O} 
\end{array}	
\right]
\]
The global sections of $\End(V_1)$ restrict to an arbitrary block upper triangular matrix at $\infty$. And being 
block upper triangular exactly means that the directrix is preserved. This proves the case for one component.

For the induction step, we are given an endomorphism $\overline{M}$ of $V_{j-1}|_{\infty}$ preserving $L^\bullet_j$,
and we want to find $M \in H^0(\End(V_{j-1})$ such that
\begin{itemize}
	\item $M|_{\infty}=\overline{M}$, and
	\item $M|_0$ is an endomorphism of $V_{j-1}|_0$ respecting $L^\bullet_{j-1}$.
\end{itemize}

This is enough for our purposes, because by induction we can choose global sections of $\End(V_i)$ for $i<j-1$ such that
all differences $D_i$ are zero, for $i<j-1$, and produce $M|_0$ over the $(j-1)$-th node.

We will choose an appropriate splitting of $V_{j-1}$ so that endomorphisms preserving the flags $L^\bullet_{j-1}$ 
and $L^\bullet_j$ become block matrices.

\begin{definition}
An ordered basis $v_1,\ldots, v_n$ \emph{generates the flag} $F^\bullet$ if each $F^i$ is the span of 
$v_1,v_2,\ldots, v_{n_i}$ for some $n_i$. 
\end{definition}
\begin{claim}\label{basis}
We may choose an ordered basis $v_1,\ldots, v_n$ of $V_{j-1}|_0$ generating $L^\bullet_{j-1}$ and 
such that $W_{j-1}|_{0}=\langle v_k,v_{k+1}, \ldots, v_n \rangle$.
\end{claim}
\begin{proof}[Proof of \cref{basis}]
We use the same idea as in the proof of \cref{lemma-transverse-flags}. Complete $L^\bullet_{j-1}$
and $G^\bullet=(0 \subset W_{j-1}|_{0} \subset V_{j-1}|_0)$ to generic complete flags $\widehat{L}^\bullet_{j-1}$ and $\widehat{G}^\bullet$.
Since $G^\bullet$ is transverse to $L^\bullet_{j-1}$ by assumption, so is
$\widehat{L}^\bullet_{j-1}$ and $\widehat{G}^\bullet$.
Choose $v_i$ such that $\langle v_i \rangle = \widehat{L}^i_{j-1} \cap \widehat{G}^{n-i-1}$.
\end{proof}

Note that the modification $L^\bullet_{j-1} \wedge (0 \subset W_{j-1}|_{0} \subset V_{j-1}|_0)$ is generated by the
ordered basis $(v_k,v_{k+1},\ldots, v_n, v_1, v_2, \ldots, v_{k-1})$.

Now pick a splitting of $V_{j-1}$ that restricts to the basis $v_1,\ldots, v_n$ over zero. Let us write the matrix $M \in H^0(\End(V_{j-1}))$ in this basis.
For convenience, we will use block notation, separating the vectors generating $W$ (that is, $v_k,\ldots, v_n$), from the remaining $v_1,\ldots, v_{k-1}$.

\[
M = 
\left[
\begin{array}{c|c}
M(V/W,V/W) & M(W,V/W) \\ \hline
M(V/W,W) & M(W,W)
\end{array}	
\right]
\]
We want $M|_0$ to respect the flag $L^\bullet_{j-1}$, and $M|_\infty$ to be the given matrix $\overline{M}$ respecting the $L_j$ flag. In this block notation this
translates into the following conditions:
\begin{enumerate}
	\item The matrix $M(V/W,V/W)$ is upper block triangular when restricted to both zero and infinity (for the same block shapes). The entries of
	$M(V/W,V/W)$ are valued in $\sh{O}_{\P^1}$, so we must take the constant matrix specified by $\overline{M}$.
	\item The same applies for the block $M(W,W)$: we want it to be upper block triangular at zero, with the restriction at infinity specified by $\overline{M}$.  Since the matrix $M(W,W)$ is valued in $\sh{O}_{\P^1}$, we must use a constant matrix.
	\item The entries of $M(W,V/W)$ are valued in $\sh{O}_{\P^1}(-1)$, so it has to be the zero matrix. This poses no problems, because at infinity the
	corresponding block of $\overline{M}$ is automatically zero to begin with.
	\item Finally, the block $M(V/W,W)$ must restrict to a specified matrix  over infinity (given by $\overline{M}$), but has to vanish over zero. This can be
	achieved, because all the entries are valued in $\sh{O}_{\P^1}(1)$, so we may take linear functions interpolating between arbitrary values 
	at zero and infinity.
\end{enumerate}

This completes the proof \cref{lemma-leftside}, and hence by induction the proof of \cref{criteria}.
\end{proof}





\section{Balanced \texorpdfstring{$F$}{F}-bundles and a maximal rank problem}
\label{sec:reduction}
The goal of this section is to translate the Main Theorem into a degree $2$ maximal rank problem for ``maximally connected chains'' in $\P^{r}$. This
will in fact allow us to establish the following slightly stronger theorem.
\begin{theorem}
\label{theorem:F-bundle}
The general cover $[\alpha\from C \to \P^1] \in \mathcal{H}_{d,g}$ has a balanced bundle of quadrics $F$ whenever $g$ can be written as $(a-1)(d-1)+bd$, for integers $a,b \geq 0$.
\end{theorem}

The Main Theorem follows from \cref{theorem:F-bundle} and the following lemma.
\begin{lemma}
Any integer $g\geq (d-3)(d-1)$ can be written as $(a-1)(d-1)+bd$ for integers $a,b\geq 0$.
\end{lemma}
\begin{proof}
We may write $g=q(d-1)+r$ for 
\[
	q \geq d-3 \text{ and } d-2\geq r \geq 0
\]

Setting $b=r\geq 0$ and $a=q-r+1\geq 0$, we have
\[
	(a-1)(d-1)+bd=(q-r)(d-1)+rd=q(d-1)+r=g
\]
as we wanted.
\end{proof}

We prove \cref{theorem:F-bundle} by first observing that being balanced is an open condition. Hence, it is enough to exhibit a single (admissible) cover with balanced $F$-bundle.  We will use \cref{criteria}
to translate this condition to a version of the maximal rank problem, specifically \cref{theorem:maximal-rank}.

The specific admissible cover we will consider arises from the following construction. Given:
\begin{itemize}
 	\item  degree $d$, genus $g_i$ simply branched covers 
 	$[\alpha_i\from C_i \to \P^1] \in \mathcal{H}_{d,g_i}$ unramified over
 	$0$ and $\infty$, for $i=1,\ldots, n$,  and
 	\item bijections $\phi_i:\alpha_i^{-1}(\infty) \to \alpha_{i+1}^{-1}(0)$, for $i=1,\ldots,n-1$ 
 \end{itemize}
 we may construct the nodal curve $X$, with irreducible components $C_1,\ldots, C_n$, obtained by identifying each $p \in \alpha_i^{-1}(0) \subset C_i$ with $ \phi(p) \in \alpha_{i+1}^{-1}(\infty) \subset C_{i+1}$ via $\phi_{i}$, for all $i$. (See \cref{figure:broken-cover}.) The curve $X$ has $d(n-1)$ nodes in total, and admits a map 
\[
\alpha \from X \to \chain{n}
\]
 to a chain of $n$ $\P^1$'s, which we denote by $\chain{n}$. The map $\alpha$ is an admissible cover in the sense of \cite{harris_kodaira_1982}. 
 
 Denote by
 $\Sigma(d;g_1,\ldots, g_n)$ the parameter space of admissible covers obtained by this procedure. Note that 
 \begin{equation}\label{eq:genus-of-admissible-cover}
 	g =p_a(X) = 1+ \sum_i (g_i -1) + d(n-1).
 \end{equation}

\begin{figure}
\begin{center}
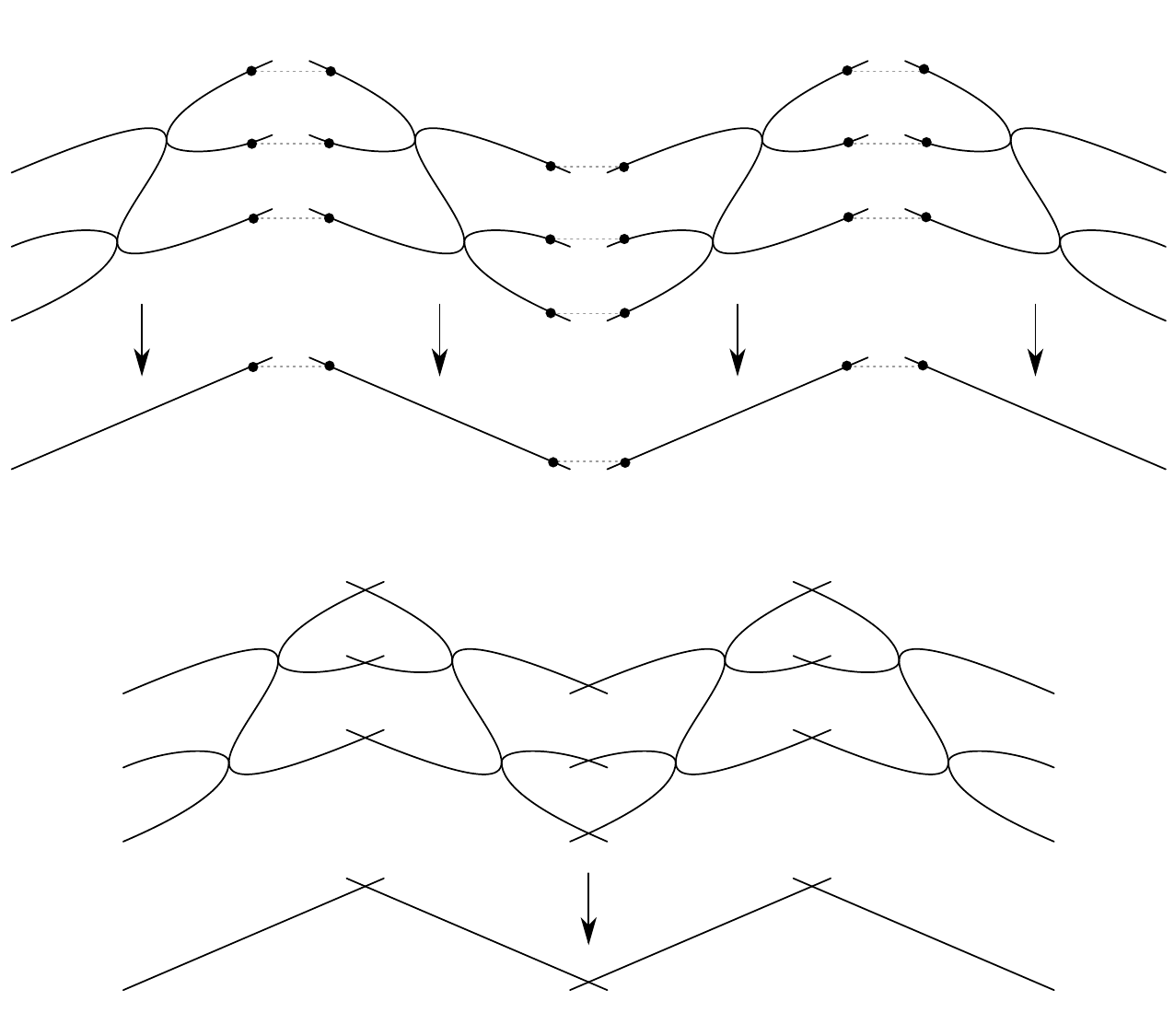
\caption{An admissible cover $[X \to \chain{n}]$  in $\Sigma(d;g_1,\ldots, g_n)$ for $n=4, d=3$.}
\label{figure:broken-cover}
\end{center}
\end{figure}

\begin{lemma}
The space $\Sigma(d;g_1,\ldots, g_n) $ is irreducible.
\end{lemma}
\begin{proof}
Consider the forgetful finite map from $g:\Sigma \to \H_{d,g_1} \times \H_{d,g_2} \times \ldots \times \H_{d,g_n}$. The target is irreducible and smooth, so we only have to show that the monodromy on the fibers is transitive. The fibers
correspond to choosing different systems of bijections $\{ \phi_i \}$. But each simply branched cover $\alpha_i: C_i \to \P^1$ has full
monodromy, which in turn induces a transitive action on the set of systems of bijections $\{ \phi_{i} \}$.
\end{proof}

 From now on, we focus our attention on proving the following proposition, which  implies \cref{theorem:F-bundle}.
\begin{proposition}
\label{proposition:admissible-cover-with-balanced-F}
If $g_1=\ldots=g_a=0$ and $g_{a+1}=\ldots=g_{a+b}=1$, then the general admissible cover $\alpha \from X \to \chain{a+b}$
in  $\Sigma(d;g_1,\ldots, g_{a+b})$ has balanced $F$-bundle.
\end{proposition}
Note that in the context of the proposition above, \cref{eq:genus-of-admissible-cover} becomes
\begin{align*}
	g&=1+ \sum_i (g_i -1) + d(a+b-1)\\
	 &= 1+ a(0-1) +b(1-1) + d(a+b-1) \\
	 &= (a-1)(d-1)+bd.
\end{align*}

We start the proof of \cref{proposition:admissible-cover-with-balanced-F} by reducing to the case where all components
are rational.
\begin{lemma}\label{lemma:reduction-to-rational}
To prove \cref{proposition:admissible-cover-with-balanced-F}, it suffices to prove the case $b=0$.
\end{lemma}
\begin{proof}
Let $\alpha:X \to \chain{a+b}$ be a general cover in $\Sigma(d;0, ...,  0,1 , ...1)$. Split it into the $a$ rational
components $X_R \to \chain{a}$ and the $b$ genus one components $X_E \to \chain{b}$.

The forgetful map $\Sigma(d;0, ...,  0,1 , ...1) \to \Sigma(d;0, ..., 0)$ (which sends the cover $X \to \chain{a+b}$ to $X_R \to \chain{a}$) is dominant. Hence, if we assume \cref{proposition:admissible-cover-with-balanced-F} is true when $b=0$, we may assume the $F$-bundle of $X_R \to \chain{a}$ is balanced.

Moreover, by \cref{item-F-elliptic} of \cref{lowG}, the $F$-bundle of $X_E \to \chain{b}$ is not only balanced, it is \emph{trivial}, up to a twist by a line bundle. Hence, the $F$-bundle of $\alpha: X\to \chain{a+b}$---which is obtained by gluing the 
(balanced) $F$-bundle of $X_R \to \chain{a}$ with the (trivial, up to twist) $F$-bundle of $X_E \to \chain{b}$---is balanced as well.
\end{proof}

Given this lemma, we need only consider the case $g_1=\ldots=g_a=0$. So we simplify notation by setting
\[
	\Sigma_{d,a} := \Sigma(d;0,0,\ldots,0).
\]

Let $\alpha:X \to \chain{a}$ be a general cover in $\Sigma_{d,a}$. We want to show that its $F$-bundle is balanced. Each component of $X$ is rational---we will denote the components by $R_1,R_2,\ldots, R_a$.

From \cref{lowG}, the Tschirnhausen bundle $E_{\alpha_{i}}$ for each component $R_{i}$ is trivial up to a twist by a line bundle.
Hence the Tschirnhausen bundle $E$ for the entire cover $\alpha \from X \to \chain{a}$ is trivial up to a twist as well, which implies $\P(E) \simeq \P^{d-2} \times \chain{a}$. Composing the relative canonical  embedding 
$X \hookrightarrow \P(E) = \P^{d-2} \times \chain{a}$ with the projection onto the first factor gives us a map $X \to \P^{d-2}$. As in the proof of \cref{lowG}, this
map is easily seen to be an embedding, and each component $R_i$ maps to a rational normal curve in $\P^{d-2}$.  The resulting embedded curve $X \subset \P^{d-2}$ is the main object of our investigation moving forward, so we make the following definition:

\begin{definition}
\label{definition:maximally-connected-chain}
A \emph{maximally connected chain of length} $n$ in $\P^r$ is a reducible nodal curve 
$X = R_1 \cup \ldots R_n \subset \P^r$ with $n$ components, which are called \emph{links}. Each 
link $R_i$ is a non-degenerate, degree $r$, nodal, connected curve, and consecutive links $R_i \cap R_{i+1}$ meet in $r+2$ points; otherwise there are no further intersections between components.  The points $R_{i} \cap R_{i+1}$ will be called \emph{anchor points}. See \cref{figure:maximally-connected-chain}.

Let $\MC_{r,n}$ denote the quasi-projective variety parametrizing all maximally connected chains of length $n$ in $\P^r$.
\end{definition}
\begin{remark}
For a generic maximally connected chain, each link $R_i$ is a smooth rational normal curve. We will eventually need to consider degenerate maximally connected chains where the $R_{i}$ become singular.
\end{remark}

\begin{figure}
\begin{center}
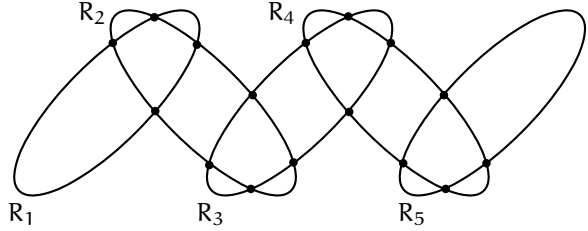
\caption{An illustration\protect\footnotemark of a maximally connected chain in $\P^2$.
}
\label{figure:maximally-connected-chain}
\end{center}
\end{figure}
\footnotetext{
Strictly speaking, 
our definition \cref{definition:maximally-connected-chain} does not allow for maximally connected chains in $\P^2$ because pairs of non-consecutive links will meet. However, for illustration purposes the real points of a chain of conics make a more informative and legible picture than a chain of twisted cubics, meeting pairwise at 5 points.}

\begin{lemma}
The space $\MC_{r,n}$  is irreducible of dimension 
$(r-1)(r+3)+(n-1)(2r+1)$.
\end{lemma}
\begin{proof}
We use induction on $n$. For $n=1$, we are simply parametrizing rational normal curves in $\P^r$ (possibly nodal), and their
parameter space is irreducible of dimension $(r-1)(r+3)$, as we wanted to show. 

For the induction 
step, notice that the space $\MC_{r,n}$ admits a dominant forgetful morphism to $\MC_{r,n-1}$ by forgetting the last 
component. It is enough to show that the fibers are irreducible of dimension $2r+1$. 
The fiber over $R_1 \cup \ldots R_{n-1} \in \MC_{r,n-1}$ is an open subset of the scheme of rational normal curves $R_n$ meeting 
$R_{n-1}$ in $r+2$ points. This is irreducible, since upon fixing the $r+2$ points of intersection, 
the space of rational normal curves through them is isomorphic to the moduli space $\overline{\mathcal{M}}_{r+2,0}$, which has dimension $r-1$. (This is the  \emph{Kapranov model} of $\overline{\mathcal{M}}_{r+2,0}$ as in \cite{kapranov_veronese_1993}.)
In total, the fiber of the forgetful map has dimension $(r+2)+(r-1)=2r+1$, which allows us to conclude the lemma.
\end{proof}

By projecting an admissible cover of type $\alpha:X \to \chain{a}$ in its relative canonical embedding $X \, \into \,\P^{d-2} \times \chain{a}$ to $\P^{d-2}$, we obtain a
maximally connected chain of length $a$ in $\P^{d-2}$. That is, there is a natural map $f :\Sigma_{d,a} \to \MC_{d-2,a}$.
Conversely, we claim that a generic
maximally connected chain can be realized as such a projection of an admissible cover in $\Sigma_{d,a}$.

\begin{lemma} \label{lemma:projection-is-dominant}
The map $f :\Sigma_{d,a} \to \MC_{d-2,a}$ is dominant.
\end{lemma}
\begin{proof}
Let $X \subset \P^{d-2}$ be a general maximally connected chain of length $n$, and let $R_{1}, \ldots , R_{a}$ be its components.  For $i = 2, \ldots ,a-1$, the rational curve $R_{i}$ has a distinguished pencil of divisors on $\cO_{R_{i}}(d)$, namely the pencil spanned by the two sets of anchor points $R_{i} \cap R_{i-1}$ and $R_{i} \cap R_{i+1}$.  Let $0 \in \P^1$ correspond to the former, and let $\infty \in \P^1$ correspond to the latter. Now, for $R_{1}$ (and $R_{a}$), simply choose a basepoint free pencil of sections of $\cO_{R_{1}}(d)$  (resp. $\cO_{R_{a}}$) containing $R_{1} \cap R_{2}$ (resp. $R_{a-1} \cap R_{a}$). The data of the curves $R_{i}$ along with the pencils of degree $d$ divisors on each gives rise to an admissible cover $\alpha \from X \to \chain{a}$. The map $X \to \P^{d-2} \times \chain{a}$ is the relative canonical embedding for $\alpha$, because each fiber $X_t \subset \P^{d-2}$ is a length $d$ arithmetically Gorenstein subscheme. Therefore, $f(\alpha)$ is the
original maximally connected chain $X \subset \P^{d-2}$, as we wanted to show.
\end{proof}

Two properties of maximally connected chains will be important for us.

\begin{definition}
A maximally connected chain $X = R_1\cup \ldots \cup R_n \subset \P^r$ is \emph{quadric-generic} if for any subchain
$Y=R_i \cup R_{i+1} \cup \ldots \cup R_{i+j}$, the restriction map 
\[
	H^0(\P^r, \sh{O}_{\P^r}(2)) \to H^0(Y,\sh{O}_Y(2))
\]
has maximal rank.
\end{definition}
\begin{definition}
Let $X=R_1 \cup \ldots \cup R_n \subset \P^r$ be a maximally connected chain. The residual intersection with $R_{i+1}$
of a quadric $Q$ containing a link $R_i$ is
\[
	\operatorname{res}_{R_{i+1}}(Q)= Q \cap R_{i+1} - R_i \cap R_{i+1}
\]
That is, the $r-2$ points of intersection of $Q$ with $R_{i+1}$ besides the $r+2$ anchor points $R_i \cap R_{i+1}$.
\end{definition}

\begin{definition} \label{definition:transverse-residues}
 A maximally connected chain $R_1 \cup \ldots \cup R_n \subset \P^r$ has \emph{transverse residues}
if either
\begin{itemize}
\item $r$ is even, or
\item $r$ is odd and
	for any subchain $Y$ of length $r+2$, 
	there are no quadrics $Q_\text{left}, Q_\text{right}$ satisfying 
	\[
		Y_\text{left} \subset Q_\text{left}, \, Y_\text{right} \subset Q_\text{right} \text{ and }\operatorname{res}_{R_\text{middle}}(Q_\text{left}) = \operatorname{res}_{R_\text{middle}}(Q_\text{right})
	\]
	where 
	\begin{itemize}
		\item $Y_\text{left}$ be the first $\frac{r+1}{2}$ links of $Y$,
		\item $R_\text{middle}$ the middle rational curve, and 
		\item $Y_\text{right}$ the remaining $\frac{r+1}{2}$ links.
	\end{itemize}
	
\end{itemize}
\end{definition}

\begin{remark}
While contrived at first sight, having transverse residues simply reflects the expectation
that the naive dimension count goes through. 
Indeed, the pair of quadrics $Q_\text{left}$ and $Q_\text{right}$ has to satisfy
\begin{itemize}
	\item $(r+2) + (r-1)\frac{r+1}{2}$ conditions for $Y_\text{left}\subset Q_\text{left}$, plus
	\item $(r+2) + (r-1)\frac{r+1}{2}$ conditions for $Y_\text{right}\subset Q_\text{right}$, plus
	\item $r-2$ conditions for $\operatorname{res}_{R_\text{middle}}(Q_\text{left}) = \operatorname{res}_{R_\text{middle}}(Q_\text{right})$.
\end{itemize}
We get a total of
\[
	2\times \left( r+2 + (r-1)\frac{r+1}{2} \right) + r-2 =r^2 +3r+1 = 2 \times \left( \binom{r+2}{2} -1 \right) +1
\]
conditions, which exceeds by one the dimensions available for choosing a pair of quadrics $(Q_\text{left},Q_\text{right})$.
\end{remark}

We may detect if the admissible cover $\alpha:X \to \chain{a}$ in $\Sigma_{d,a}$ has balanced
$F$-bundle only using properties of the maximally connected $f(\alpha) \in \MC_{d-2,a}$.

\begin{lemma} \label{lemma:reduction-to-quadric-generic}
A cover $\alpha:X \to \chain{a}$ in $\Sigma_{d,a}$ has balanced $F$-bundle if and only if 
$f(\alpha) \in \MC_{d-2,a}$ is quadric-generic and has transverse residues.
\end{lemma}

\Cref{proposition:admissible-cover-with-balanced-F} -- and hence \cref{theorem:F-bundle} -- follows from 
\cref{lemma:projection-is-dominant,lemma:reduction-to-quadric-generic,lemma:reduction-to-rational} and the following claims, which will be proved in \cref{sec:the_maximal_rank_problem_for_maximally_connected_chains}.
\begin{claim} \label{claim:general-maximally-connected-chain-is-quadric-generic}
A general maximally connected chain is quadric-generic.
\end{claim}
\begin{claim}\label{claim:general-maximally-connected-chain-has-transverse-residues}
 A general maximally connected chain has transverse residues.
\end{claim}

\begin{proof}[Proof of \cref{lemma:reduction-to-quadric-generic}] 
\label{sub:proof_of_cref_lemma_reduction_to_quadric_generic}
We will first translate the $F$-bundle of $\alpha:X \to \chain{a}$ into
the language of quadrics in $\P^{d-2}$, and then use \cref{criteria}
to translate the balancedness of $F_{\alpha}$ into the statement that $f(\alpha)$ is quadric-generic and has transverse residues.

Start by recalling the geometric interpretation of the directrix of the $F$-bundle for
genus zero covers.
\begin{lemma}
\label{lemma:description-of-directrix}
There is a map of vector bundles over $\chain{a}$: 
\[
\Phi \from F \to \chain{a} \times H^0(\P^{d-2}, \cO_{\P^{d-2}}(2))
\]
such that:

\begin{enumerate}
\item $\Phi_t$ sends the vector space $F_t$ to the vector space of quadrics vanishing at the
$d$ points $\alpha^{-1}(t) \subset X \subset \P^{d-2}$. 
\item For each component $R_i \subset X$, the directrix $W_i \subset F_i$ maps under $\Phi_t$ to the
vector space $H^{0}(I_{R_{i}}(2))$ of quadrics containing the rational normal curve $R_i$.
\end{enumerate}
\end{lemma}
\begin{proof}
This is contained in the proof of \Cref{lowG}, part $(b)$.
\end{proof}

Our analysis now breaks up into two cases, depending on the parity of the degree $d$. Recall the language
of left and right flags introduced in \cref{subsub-the-two-natural-filtrations}.

\begin{lemma}\label{lemma:flag-description-even}
Let $\alpha:X \to \chain{a} \in \Sigma_{d,a}$ be given. Assume that $d$ is even and $f(\alpha)=R_1 \cup \ldots \cup R_a$ is quadric-generic. For any node $p_{j} \in \chain{a}$, the 
left flag $L^\bullet_j=(L^1_j\subset L^2_j \subset \ldots \subset L^k_j)$ of $F_\alpha$ may be described as follows:
\begin{enumerate}
	\item The total space: $L_j^k$ is the vector space of quadrics in $\P^{d-2}$ containing the $d$ anchor 
	points over the $j$-th node.
	\item $L_j^{k-l}$ is the vector space of quadrics in $\P^{d-2}$ containing the sub-chain $R_{j-l+1} \cup \ldots \cup R_j$.
    \item The codimension of $L^{k-l}_j \subset L^k_j$ is $l(d-3)$. 
    \item The length $k$ of the flag $L^\bullet_j$ is equal
to the minimum of $j+1$ and $\frac{d}{2}$.
\end{enumerate}
\end{lemma}
\begin{proof}
The proof is by induction on $j$. For $j=1$, the description above follows directly from the definition of the flag
and \cref{lemma:description-of-directrix}.
Supposing the description is valid for $j$, we now prove its validity for $j+1$.

The right hand modification of the flag $L^\bullet_j$ consists of three types of spaces:
\begin{itemize}
	\item Intersections of $L^{k-l}_j :=  \{ \text{quadrics containing $R_{j-l+1},\ldots, R_{j}$}\}$ with the directrix $W_{j+1}: = \{\text{quadrics containing $R_{j+1}$}\}.$ That is, the quadrics containing the range 
	from $R_{j-l+1}$ up to $R_{j+1}$. This has the expected dimension since $f(\alpha)$ is quadric-generic.
	\item The directrix $W_{j+1}$ itself, that is, quadrics containing $R_{j+1}$.
	\item The spans $L^i_j+W_{j+1}$. But these will all be equal to the total space
	of quadrics containing the $d$ anchor points. Indeed, $W_{j+1}$ has codimension $d-3$, and the smallest space $L^1_j$
	has dimension at least $d-3$, by induction.
	If the span  $W_{j+1} +L^1_j$ were a proper subspace of $F|_{p_{j+1}}$, then the 
	intersection $L^1_j \cap W_{j+1}$ would have dimension larger than expected, which would violate being 
	quadric-generic.
\end{itemize}
\end{proof}

From \cref{criteria} and  \cref{lemma:flag-description-even} it follows straightforwardly that if $d$ is even and $f(\alpha)$ is quadric-generic 
then $\alpha$ has a balanced $F$-bundle. Conversely, if $f(\alpha)$ is not quadric-generic, it is easy to
see that the bundle of quadrics of $\alpha$ does not have transverse-flags, and therefore is not balanced. This
settles \cref{lemma:reduction-to-quadric-generic} for even degrees $d$.

Here is the analogue of \cref{lemma:flag-description-even} for odd degrees.
\begin{lemma}\label{lemma:flag-description-odd}
Let $d$ be odd and $f(\alpha)=R_1 \cup R_2 \cup \ldots \cup R_a$ be quadric-generic and have transverse residues.
The elements of the left flag $L_j^\bullet$ are of the following form:
\begin{enumerate}
\item The total space: Quadrics containing the anchor points $p_1,\ldots, p_d$ over the
$j$-th node. This is a vector space of dimension $d(d-3)/2$.
\item Quadrics containing a subchain of rational curves
$R_{j-l+1} \cup \ldots \cup R_j$. This has dimension $(d-2l) \frac{d-3}{2}$.
Note that $l< \frac{d}{2}$.
\item Quadrics containing the subchain $R_{j-l+1} \cup \ldots \cup R_j$, and whose
residual intersection with $R_{j-l}$ is equal to the residual intersection of a quadric
containing the length $\frac{d-1}{2}$ subchain
$R_{j-l-\frac{d-1}{2}} \cup \ldots \cup R_{j-l-1}$. We will call this the \emph{restriction condition on $R_{j-l}$}.
This flag element has dimension
$(d-2l-1)\frac{d-3}{2}$.
\end{enumerate}
\end{lemma}
The proof is analogous to the proof of \cref{lemma:flag-description-even}.
The third case comes from the span construction.

To establish balancedness of the $F$-bundle, \cref{criteria} tells us we need to establish transverse directrices.
This amounts to checking the following three conditions:
\begin{itemize}
	\item The linear series of quadrics containing  an arbitrary subchain $R_i \cup \ldots \cup R_j$ has the expected dimension.
	This follows from $f(\alpha)$ being quadric-generic.
	\item The linear series of quadrics containing  an arbitrary subchain $R_i \cup \ldots \cup R_j$
	and satisfying the restriction condition on $R_{i-1}$ has the expected dimension.
	In many situations, this follows from being quadric-generic , since the restriction map
	\[
		H^0(\P^n, \sh{I}_{R_i \cup \ldots \cup R_j}(2)) \to 
		H^0(R_{i-1},\sh{I}_{\text{anchor}}(2))
	\]
	is surjective as long as the length of the range $R_{i-1} \cup \ldots \cup R_j$ is at most equal to $\frac{d-1}{2}$. Indeed, the kernel corresponds to quadrics
	containing the range from $i-1$ to $j$, which we know the dimension of by being quadric-generic.

	The only extremal case to check is when the subchain $R_i \cup \ldots \cup R_j$ has length
	$\frac{d-1}{2}$. In this situation, the property we require is precisely having transverse residues (see \cref{definition:transverse-residues}).

	\item The linear series of quadrics containing 
	an arbitrary subchain $\bar{X}=R_i \cup \ldots \cup R_j$ and satisfying the restriction condition
	on both $R_{i-1}$ and $R_{j+1}$ has the expected dimension. Again, the restriction map
	\[
		H^0(\P^n, \sh{I}_{R_i \cup \ldots \cup R_j}(2)) \to 
		H^0(R_{i-1},\sh{I}_{\text{anchor}}(2)) \oplus H^0(R_{j+1},\sh{I}_{\text{anchor}}(2))
	\]
	will be surjective for subchains $\bar{X}$ of length less than $\frac{d-3}{2}$, since we can control
	the dimension of the kernel by being quadric-generic. This implies that the restriction conditions on
	$R_{i-1}$ and $R_{j+1}$ are independent.

	On the other hand, if the length of the
	subchain $\bar{X}$ is bigger than $\frac{d-3}{2}$, there are no such quadrics by the previous item. Hence,
	the only case left to consider is subchains $\bar{X}$ of length exactly $\frac{d-3}{2}$.

	In this case, let $V$ be the space of quadrics containing $\bar{X}$ and satisfying the
	restriction condition on $R_{i-1}$. The space $V$ has the expected dimension $d-3$, by the previous item.
	Moreover, the restriction map
	\[
		V  \to 
		H^0(R_{j+1},\sh{I}_{\text{anchor}}(2))
	\]
	is an isomorphism, since a non-zero kernel would violate having transverse residues. Therefore, the linear
	series of quadrics containing the subchain $R_i \cup \ldots \cup R_j$ and satisfying the restriction
	condition on both $R_{i-1}$ and $R_{j+1}$ has the expected dimension, as we wanted to show.
\end{itemize}
Furthermore, the argument above shows that for odd degrees $d$, whenever $f(\alpha)$ is quadric-generic and has transverse residues, the
$F$-bundle of $\alpha$ is balanced. Conversely, if one of these conditions fail, we can show that the $F$-bundle does 
not have transverse directrices, and therefore is not balanced.
This completes the proof of \cref{lemma:reduction-to-quadric-generic}.
\end{proof}

\section{The maximal rank problem for maximally connected chains} 
\label{sec:the_maximal_rank_problem_for_maximally_connected_chains}
The previous section demonstrated that \cref{claim:general-maximally-connected-chain-is-quadric-generic,claim:general-maximally-connected-chain-has-transverse-residues} were exactly what was necessary to conclude \cref{theorem:F-bundle} (and the Main Theorem as a corollary).
In this section, we provide these necessary facts. The key result is the following.
\begin{theorem}
\label{theorem:maximal-rank}
For a general maximally connected chain $X \in \MC_{r,n}$, the restriction map
\[
	\rho: H^0(\P^r, \sh{O}_{\P^r}(2)) \to H^0(X, \sh{O}_X(2))
\]
has maximal rank.
\end{theorem}
Note that \cref{claim:general-maximally-connected-chain-is-quadric-generic} follows from \cref{theorem:maximal-rank}
and the fact that the forgetful map

\[
	\operatorname{subchain}_{n_1,n_2}:\MC_{r,n} \to \MC_{r, n_2-n_1+1}
\]
sending $R_1 \cup \ldots \cup R_n$ to the subchain $R_{n_1} \cup R_{n_1+1} \cup \ldots \cup R_{n_2}$ is dominant.

We will defer the proof of \cref{claim:general-maximally-connected-chain-has-transverse-residues} to \cref{sub:transverse_residues}, since the argument will seem more natural after studying the proof of \cref{theorem:maximal-rank}.

We can rephrase \cref{theorem:maximal-rank} as a way to determine how many quadrics contain a general maximally connected chain.
\begin{lemma}
\label{lemma:euler-char}
For any maximally connected chain $X \subset \P^r$ with $r \geq 3$, we have
\[
	h^0(\sh{I}_X(2))-h^1(\sh{I}_X(2)) = \frac{(r-1)(r+2-2n)}{2}.
\]
\end{lemma}
\begin{proof}
Consider the standard short exact sequence
\[
	0 \to \sh{I}_X(2) \to \sh{O}_{\P^r}(2) \to \sh{O}_{X}(2) \to 0.
\]
Since $H^1(\sh{O}_{\P^r}(2))=0$, the long exact sequence gives
\begin{align*}
	h^0(\sh{I}_X(2))-h^1(\sh{I}_X(2)) &= h^0(\sh{O}_{\P^r}(2)) - h^0(\sh{O}_{X}(2)) \\
	&= \binom{r+2}{2} - (2nr-(n-1)(r+1)+1) +h^1(\sh{O}_X(2)) \\
	&= \frac{(r-1)(r+2-2n)}{2} +h^1(\sh{O}_X(2)).
\end{align*}
We are left with showing that $h^1(\sh{O}_X(2))=0$. 

By Serre duality, 
$h^1(\sh{O}_X(2)) = h^0(\omega_X(-2))$. The line bundle $\omega_X(-2)$
has degree $-r$ on $R_1$ and $R_n$, and degree $2$ on the other components.
Let $\sigma$ be a global section of $\omega_X(-2)$.
Since $-r<0$, it vanishes on $R_1$, and hence also on $R_1 \cap R_2$. So 
the section restricted to $R_2$ has at least $r+2$ zeros, but degree $2$. Hence it
vanishes identically on $R_2$ as well, and so on for each component. Therefore, $\sigma=0$,
as we wanted to show.
\end{proof}

Taking \cref{lemma:euler-char} into account, \cref{theorem:maximal-rank} says that 
\begin{equation}\label{dim:quad}
	h^0(\sh{I}_X(2))= \begin{cases}
	0, \text{ if }  n \geq \frac{r+2}{2},\\
	\frac{(r-1)(r+2-2n)}{2}, \text{ else.}
	\end{cases}
\end{equation}

\begin{proof}[Proof of \cref{theorem:maximal-rank}]
We will use induction on the ambient dimension $r$. The base case $r=3$ translates into the following three simple facts:
\begin{itemize}
	\item A twisted cubic is contained in a three dimensional linear system of quadrics.
	\item There is a unique quadric containing a pair of twisted cubics meeting at five points.
	\item By picking the third link $R_{3}$ generically,  the resulting length 3 maximally connected chain $X$ of twisted cubics is not contained in any quadric.
\end{itemize}

\begin{figure}
\begin{center}
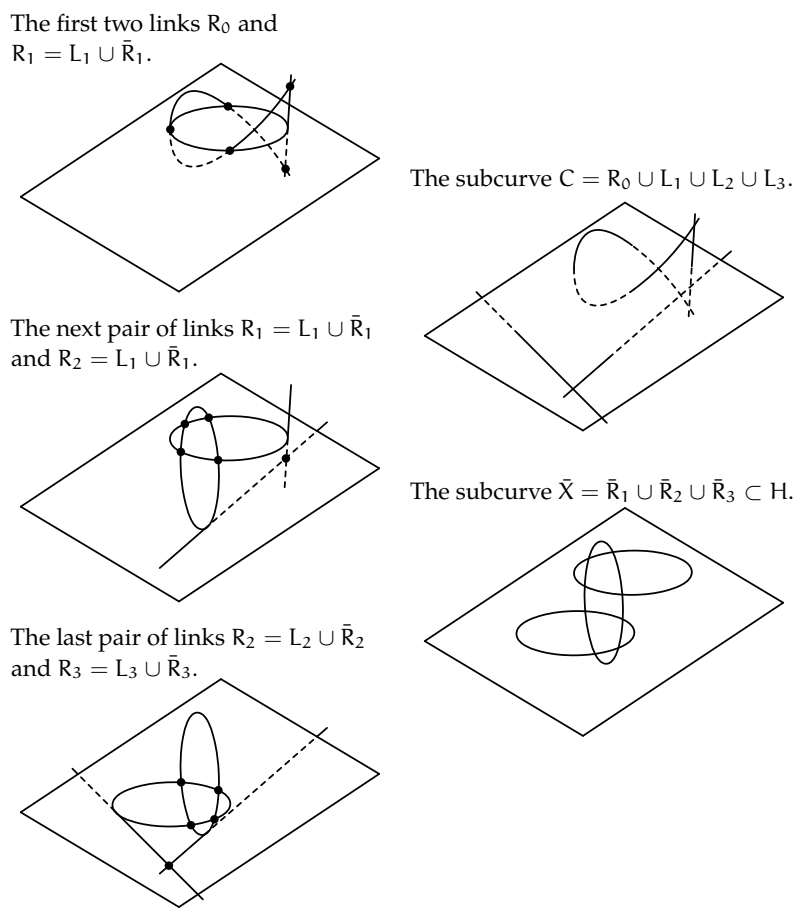
\caption{The maximally connected chain $X$ used in the induction argument.}
\label{figure:erics-curve}
\end{center}
\end{figure}

For the induction step, we use the following key construction suggested to us by Eric Larson:
\begin{construction}
\label{construction:eric}
$X = C \cup \overline{X} = R_0 \cup R_1 \cup \ldots \cup R_{n-1}$ is a maximally connected chain, constructed as follows.
\begin{itemize}
	\item Take $R_0 \subset \P^r$ a smooth rational normal curve.
	\item Choose a general hyperplane $H \subset \P^r$.
	\item Let $L_1$ be a general 2-secant line to $R_0$.
	\item Let $\overline{R}_1 \subset H$ be a rational normal curve (of degree $r-1$) containing the $r$ points $R_0 \cap H$ and the point $p : = L_1 \cap H$. Set $R_1=L_1\cup_p \overline{R}_1$.
	\item Inductively, choose $\overline{R}_{i+1} \subset H$ a general rational normal curve meeting $\overline{R}_i$ in
	$r+1$ general points. That is, the curve $\overline{X} := \overline{R}_1 \cup \overline{R}_2 \cup \ldots \cup \overline{R}_{n-1}$ is a maximally connected chain of length $n-1$ in $H= \P^{r-1}$.
	\item Inductively choose $L_i$ to be a general line joining $L_{i-1}$ and $\overline{R}_i$. Set
	$R_i= L_{i} \cup \overline{R}_i$.
	\item Let $C = R_0 \cup L_1 \cup L_2 \cup \ldots \cup L_{n-1}$.
\end{itemize}
See \cref{figure:erics-curve} for a diagrammatic representation of this construction. 
\end{construction}

We will show that $X$ as above satisfies the maximal rank condition for quadrics. That is, we will show: 
\begin{itemize} 
\item If $n> \frac{r+2}{2}$, then there are no quadrics containing $X$. 
\item Otherwise there is a $\frac{(r-1)(r+2-2n)}{2}$ dimensional space of quadrics containing $X$.
\end{itemize}
We do this by looking at the sequence
\[
	0 \to \sh{O}_{\P^r}(1)= \sh{I}_H(2) \to \sh{O}_{\P^r}(2) \to \sh{O}_H(2) \to 0
\]
and tensoring it with $\sh{I}_X$ to get
\begin{equation}\label{quad:seq}
	0 \to \sh{I}_{C \subset \P^r}(1) \to \sh{I}_{X \subset \P^r}(2) \to \sh{I}_{\overline{X} \subset H}(2) \to 0.
\end{equation}
 By induction we know that $\overline{X} \subset H$ satisfies the maximal rank condition, and
 the sequence \eqref{quad:seq} will let us understand the quadrics containing $X$.

\subsubsection*{The case $n > \frac{r+2}{2}$} 
\label{ssub:large-n}
Then $n-1 \geq \frac{(r-1)+2}{2}$ and by induction and \eqref{dim:quad}, we can assume that 
$h^0(\sh{I}_{\overline{X} \subset H}(2) )=0$. Hence, 
\[
	h^0(\sh{I}_{X \subset \P^r}(2) ) = h^0(\sh{I}_{C\subset \P^r}(1))
\]
and the latter is zero, since $C$ is not contained in any hyperplane (the component $R_0$ is non-degenerate). Hence, $h^0(\sh{I}_{X \subset \P^r}(2) )=0$, as we wanted to show.
\subsubsection*{The case $n \leq \frac{r+2}{2}$} 
We have $n-1 < \frac{(r-1)+2}{2}$, and by induction 
we can assume
\[
	h^0(\sh{I}_{\overline{X} \subset H}(2)) = \frac{(r-2)(r+3-2n)}{2} = \frac{(r-1)(r+2-2n)}{2} +n-2.
\]
Hence, we want to show that the inclusion
\begin{equation}\label{inclusion}
	H^0(\sh{I}_{X \subset \P^r}(2)) \subset H^0(\sh{I}_{\overline{X} \subset H}(2))
\end{equation}
has codimension $n-2$. The inclusion \eqref{inclusion} factors as
\begin{equation}\label{inclusion2}
	H^0(\sh{I}_{X \subset \P^r}(2)) \subset H^0(\sh{I}_{\overline{X} \cup R_0 \cup L_1}(2)) \subset H^0(\sh{I}_{\overline{X} \subset H}(2))	
\end{equation}
and the latter inclusion is actually an isomorphism, since it fits in the sequence
\[
	H^0(\sh{I}_{R_0\cup L_1}(1)) \to H^0(\sh{I}_{\overline{X} \cup R_0 \cup L_1}(2)) \to H^0(\sh{I}_{\overline{X} \subset H}(2))	\to H^1(\sh{I}_{R_0\cup L_1}(1))
\]
where the first and last groups vanish, as deduced from the sequence
\[
\xymatrix{
	0\ar[r] & H^0(\sh{I}_{R_0\cup L_1}(1)) \ar[r]&  H^0(\sh{O}_{\P^r}(1)) \ar[r]& 
	 H^0(\sh{O}_{R_0 \cup L_1}(1))\ar[dll]	&\\ & H^1(\sh{I}_{R_0\cup L_1}(1)) \ar[r]& 
	H^1(\sh{O}_{\P^r}(1))=0	}
\]
and the fact that the map $H^0(\sh{O}_{\P^r}(1)) \to
H^0(\sh{O}_{R_0 \cup L_1}(1))$ is an isomorphism (in other words, $R_0 \cup L_1$ is linearly
normal, i.e., it is embedded with the full linear series).

So we have translated our problem into showing that the first inclusion in \eqref{inclusion2},
that is	$H^0(\sh{I}_{X \subset \P^r}(2)) \subset H^0(\sh{I}_{\overline{X} \cup R_0 \cup L_1}(2))$,
has codimension $n-2$. But this factors naturally in a sequence of $n-2$ inclusions:
\begin{equation}
\label{eq:chain-of-inclusions}
\begin{split}
	H^0(\sh{I}_{X}(2))=H^0(\sh{I}_{\overline{X} \cup R_0 \cup L_1 \cup L_2 \cup \ldots L_{n-1}}(2)) 
	&\subset H^0(\sh{I}_{\overline{X} \cup R_0 \cup L_1 \cup L_2 \cup \ldots \cup L_{n-2}}(2)) \\
	&\subset \ldots\\
	&\subset H^0(\sh{I}_{\overline{X} \cup R_0 \cup L_1 \cup L_2}(2))  \\
	&\subset H^0(\sh{I}_{\overline{X} \cup R_0 \cup L_1}(2)) 
\end{split}
\end{equation}

So it is enough to show the following.
\begin{proposition}
\label{claim:strict-inclusion}
Each inclusion in \eqref{eq:chain-of-inclusions} is strict. 
In other words, 
for each $i=2, \ldots, n-1$, there exists a quadric containing 
$\overline{X} \cup R_0 \cup L_1 \cup \ldots \cup L_{i-1}$,
which does not contain the line $L_i$.
\end{proposition}

 We establish \cref{claim:strict-inclusion} below, completing the proof of
 \cref{theorem:maximal-rank}.
\end{proof}


The key lemma needed for the proof of \cref{claim:strict-inclusion}  is the following:
\begin{lemma}
\label{lemma:induction-step-join}
Let $L \subset \P^r$ be a line, and $H\subset \P^r$ a 
hyperplane transverse to $L$, and $R \subset H$ a rational normal curve.
Let $V \subset H^0(\sh{I}_{L \cup R}(2))$ be a non-empty linear
series. Assume:
\begin{enumerate}[label=(\alph*)]
	\item \label{condition:irreducible-quadric} no element of $V$ is a reducible quadric containing $H$,
	\item the general element $Q_\text{gen} \subset \P^r$ of $V$ is such that
	$L \not\subset \Sing Q_\text{gen}$.
\end{enumerate}
Then for general choices of $p\in L$ and $q \in R$:
\begin{itemize}
\item The line $L_\text{gen} = \overline{pq}$
is not contained in $Q_\text{gen}$. 
\item Among the quadrics $Q \in V$ containing
$L_\text{gen}$, a general one does not contain $L_\text{gen}$ in its singular locus.
\end{itemize}
\end{lemma}

\begin{proof}[Proof of \cref{claim:strict-inclusion} assuming \cref{lemma:induction-step-join}]
We want to show that each inclusion in \eqref{eq:chain-of-inclusions} is strict.
Let us start with the last one. We set $V= H^0(\sh{I}_{\overline{X}\cup R_0 \cup L_1}(2))$, $L=L_1$, and
$R=\overline{R}_2$. None of the elements of $V$ are reducible quadrics, because they contain the non-degenerate curve $R_0$. Thus, condition \ref{condition:irreducible-quadric} is met.

Let us check that the general quadric in $V$ is not singular all along
$L_1$. It is enough to show that the restriction
of the quadric to $H$ is smooth at $L_1 \cap H$. The restriction of the linear system
$H^0(\sh{I}_{\overline{X}\cup R_0 \cup L_1 \subset \P^r}(2))$ to $H$ is 
$H^0(\sh{I}_{\overline{X} \subset H}(2))$, which only depends on $\overline{X}$. The general quadric in $H$ containing $\overline{X}$ cannot be singular all along $\overline{R}_1$, because
$\overline{R}_1$ is a non-degenerate curve in $H$, and the singular locus of a quadric is always a linear space. Hence, as long as we pick $L_1 \cap H$ to be a general point in $\overline{R}_1$, we will be fine. We can do this, as a matter of fact in \cref{construction:eric} we may first choose $\overline{X}$ and any $r+1$ distinct points in $\overline{R}_1$, and then we can find an $L_1 \cup R_0 \subset \P^r$
meeting $H$ exactly at these points.

The hypotheses of \cref{lemma:induction-step-join} are satisfied, so its conclusion says that
for a general choice of $L_2$, the inclusion \eqref{eq:chain-of-inclusions} will be strict.

Moreover, \cref{lemma:induction-step-join} also says that a general quadric in  
$H^0(\sh{I}_{\overline{X}\cup R_0 \cup L_1 \cup L_2}(2))$ is not singular all along $L_2$. Thus we conclude \cref{claim:strict-inclusion} by iteratively applying \cref{lemma:induction-step-join}.

\end{proof}

Hence, we will be done with the proof of \cref{theorem:maximal-rank} as soon as 
we establish \cref{lemma:induction-step-join}.

\begin{proof}[Proof of \cref{lemma:induction-step-join}]
For the first part, assume that $L_\text{gen}\subset Q_\text{gen}$. Then $Q_\text{gen}$ contains 
the join $J(L,R)$. But the tangent space to $J(L,R)$ at $p \in L$
is $r$-dimensional, because the cone $C_{p}R$ with vertex $p$ over $R$ is contained in
$J(L,R)$. Hence, $Q_\text{gen}$ is singular along $L$ as well, but we assumed this was not the case -- contradiction.

Now let us show that among the quadrics in $V$ containing $L_\text{gen}=\overline{pq}$, a general
one does not contain $L_\text{gen}$ in its singular locus. Pick $p \in L$ such that 
the general element of $V$ is not singular at $p$ (this can be done by our assumption). Now there are two cases.
\begin{description}
\item[Case 1] One possibility is that as we vary $q\in R$, we get a non-trivial family of
hyperplanes consisting of quadrics in $V$ containing $\overline{pq}$. But then the union of these hyperplanes
is the whole vector space $V$, and we know that the general quadric in $V$ is smooth at $p$. Hence, it is smooth
at the general point of the line $\overline{pq}$ as well.

\item[Case 2] Otherwise, every quadric in $V$ containing $\overline{pq}$ also contains the cone $C_pR$. 
Pick any of these quadrics, $Q$. Then $Q$ cannot be singular all along $R$, because the
singular locus of $Q$ is a linear space, and $R$ spans the hyperplane $H$. Hence, for some
$q \in R$, the quadric $Q$ is smooth at $q$. Take the line $\overline{pq}$  which is contained in $C_{p}R$, and hence in $Q$. As $Q$ is smooth at $q$, it is smooth at the general point of 
$\overline{pq}$.
\end{description}
In any case, we still conclude \cref{lemma:induction-step-join}.
\end{proof}

\subsection{Transverse residues} 
\label{sub:transverse_residues}

\newcommand{\Xright}{X_{\text{right}}}
\newcommand{\Hright}{H_{\text{right}}}
\newcommand{\Lambdaright}{\Lambda_{\text{right}}}

\newcommand{\Hleft}{H_{\text{left}}}
\newcommand{\Lambdaleft}{\Lambda_{\text{left}}}
\newcommand{\Xleft}{X_{\text{left}}}

\newcommand{\Rmiddle}{R_{\text{middle}}}

We now adapt \cref{construction:eric} to prove \cref{claim:general-maximally-connected-chain-has-transverse-residues}.

\begin{proof}[Proof of \cref{claim:general-maximally-connected-chain-has-transverse-residues}]
From now on the ambient dimension $r$ will be odd. We want to show that a general maximally connected chain $X\subset \P^r$ of length $r+2$ has transverse residues, as in \cref{definition:transverse-residues}. This means we
partition $X$ in three groups: the middle link $\Rmiddle$, and the $\frac{r+1}{2}$ links coming before ($\Xleft$) and after ($\Xright$) it. We want to show that a pair of
quadrics $Q_{\text{left}}$ and $ Q_{\text{right}}$ containing $\Xleft$ and $\Xright$ respectively cannot have the same residual intersections with $\Rmiddle$.

It is enough to exhibit a single example of such an $X$. Consider an $X$ with the following properties:
\begin{itemize}
    \item $\Rmiddle$ is a smooth rational normal curve
    \item $\Rmiddle \cup \Xright$ and $\Xleft \cup \Rmiddle$ are exactly as in \cref{construction:eric}. That is, to construct $\Xright$, choose a hyperplane $H$ and let $R_i=L_i \cup \overline{R}_i$, where:
    \begin{itemize}
        \item $L_1$ is a general secant line to $\Rmiddle$
        \item $\overline{R}_1$ is a general rational normal curve in the hyperplane $H$ containing $\Rmiddle \cap H$ and $L_1 \cap H$
        \item $\overline{R}_1\cup \overline{R}_2\cup \overline{R}_3 \cup \ldots \cup \overline{R}_{\frac{r+1}{2}}\subset H$ form a 
        general maximally connected chain in $H$.
        \item $L_{i}$ is a general line joining $L_{i-1}$ and $\overline{R}_{i}$ for $i=2,\ldots,\frac{r+1}{2}$.
    \end{itemize}
\end{itemize}

 We claim that such a curve $X$ has transverse residues.    Let $\Lambdaright$ be the linear span of $L_1,L_2,\ldots, L_{\frac{r+1}{2}}$. Any quadric containing $\Xright$ must split as $H \cup \Hright$, where $\Hright$ is a hyperplane containing $\Lambdaright$. Otherwise, the restriction of such quadric to $H$ would contain a general maximally connected chain of length $\frac{r+1}{2}$, which \cref{theorem:maximal-rank} says cannot exist.

    The $r-2$ points of residual intersection of the reducible quadric $ H\cup \Hright $  with $\Rmiddle$ is the set $\Rmiddle\cap \Hright$ minus the pair of points in $L_1 \cap \Rmiddle= \Lambdaright \cap \Rmiddle$. Therefore, to see how the residual points vary in $\Rmiddle$, we only need to remember the linear space $\Lambdaright$.

    Define $\Lambdaleft$ analogously. Note that we may choose the pair of
    $(\Lambdaright, \Lambdaleft)$ arbitrarily, as long as they meet $\Rmiddle$ in a pair of points each. This is because when performing \cref{construction:eric}, we may first choose the lines $L_i$ arbitrarily, and then choose fitting $\overline{R}_i$.

So we want to show that for $\Lambdaright$ and $\Lambdaleft$ generic, there are no hyperplanes $\Hleft$ and $\Hright$ containing the respective linear spaces $\Lambdaleft, \Lambdaright$ whilst having the same residual intersection on $\Rmiddle$. Consider the incidence correspondence:
\begin{align*}
\Sigma := \big\{ (\Lambdaleft,\Lambdaright&, \Hleft,\Hright, p_1,\ldots ,p_{r-2}) \mid\\
& \#(\Lambdaleft \cap \Rmiddle)=\#(\Lambdaright \cap \Rmiddle)=2, \\
& \Lambdaleft \subset \Hleft, \, \Lambdaright \subset \Hright, \\
& \{p_1,\ldots ,p_{r-2}\} \cup (\Lambdaleft \cap \Rmiddle) = \Hleft \cap \Rmiddle, \\
& \{p_1,\ldots ,p_{r-2}\} \cup (\Lambdaright \cap \Rmiddle) = \Hright \cap \Rmiddle
\big\}
\end{align*}
 That is, the points $p_i$ are the (common) residual intersection of the hyperplanes $\Lambda\subset H$ with
 $\Rmiddle$. The incidence correspondence $\Sigma$ admits a projection map to: 
 \[
 \Sigma'= \big\{(\Lambdaleft, \Lambdaright) \mid
 \#(\Lambdaleft \cap \Rmiddle)=\#(\Lambdaright \cap \Rmiddle)=2
 \big\}
 \]

 We want to show that the map $\Sigma \to \Sigma'$ is not dominant. We do this by showing that $\dim \Sigma < \dim \Sigma'$, as follows.
\begin{align*}
\dim \Sigma = & \,r-2 & \text{ (choose the $p_i$'s)}\\
& +2 + 2 & \text{ (choose $\Lambdaleft\cap \Rmiddle$, $\Lambdaright \cap \Rmiddle$)} \\
& + 0 + 0 & \text{ ($\Hleft, \Hright$ are the span of $p_i$'s and $\Lambda\cap \Rmiddle$)}\\
& + \frac{(r-3)(r-1)}{4} & \text{ (choose $\Lambdaright\subset \Hright$ containing $\Lambdaright \cap \Rmiddle$)}\\
&+\frac{(r-3)(r-1)}{4} & \text{ (same for $\Lambdaleft$)}\\
 =& \frac{(r-1)^2}{2} +3 &
\end{align*}
while
\begin{align*}
\dim \Sigma' = &\, 2+2  & \text{ (choose pair points $\Lambda \cap \Rmiddle$)}\\
& +\frac{(r-1)^2}{4} & \text{ (choose $\Lambdaright$ containing $\Lambdaright \cap \Rmiddle$)}\\
&+\frac{(r-1)^2}{4} & \text{ (same for $\Lambdaleft$)}\\
 = &\frac{(r-1)^2}{2}+4\\
 =&\dim \Sigma+1 &
\end{align*}
which is as we desired. This completes the proof of \cref{claim:general-maximally-connected-chain-has-transverse-residues}.
\end{proof}



\bibliographystyle{amsalpha}
\bibliography{GenericCasnatiEkedahl}

\end{document}